\documentclass[leqno,12pt]{amsart}
\usepackage{amssymb}
\usepackage{amsmath}
\usepackage{enumerate}
\usepackage{amsfonts}
\usepackage{hyperref}
\usepackage{mathrsfs}
\usepackage{tikz}
\usepackage{mathrsfs}

\usepackage[a4paper, margin=2.5cm]{geometry}

\hypersetup{
    colorlinks=true,
    linkcolor= blue,
    citecolor =cyan,
    urlcolor = teal,
}

\newtheorem{theorem}{Theorem}[section]
\newtheorem{corollary}[theorem]{Corollary}
\newtheorem{lemma}[theorem]{Lemma}
\newtheorem{proposition}[theorem]{Proposition}

\newtheorem{definition}[theorem]{Definition}

\numberwithin{equation}{section}

\newcommand{\supp}{\text{\rm supp}\,}

\begin{document}
\title[Riesz transforms and Hardy spaces
for Dunkl--Schr\"odinger operators]{Riesz transform characterization of Hardy spaces
associated with Dunkl--Schr\"odinger operators}

\author[Agnieszka Hejna-Łyżwa]{Agnieszka Hejna-Łyżwa}

\address{Agnieszka Hejna-Łyżwa, Uniwersytet Wroc\l awski,
Instytut Matematyczny,
Pl. Grunwaldzki 2,
50-384 Wroc\l aw,
Poland}
\email{hejna@math.uni.wroc.pl}

\subjclass[2020]{Primary 42B30; Secondary 42B20, 47D08, 35J10}

\keywords{Dunkl operators, Schr\"odinger operators, Hardy spaces,
Riesz transforms, atomic decomposition, heat kernel estimates}

\begin{abstract}
Let $L=-\Delta_k+V$ be a Schr\"odinger operator associated
with the Dunkl Laplacian and a nonnegative potential.
We characterize the Hardy space $H_L^1$, defined by the
maximal operator of the semigroup generated by $-L$,
in terms of the Riesz transforms associated with $L$.
Under a local integral condition on the potential and
a decay condition on the semigroup, formulated relative
to an admissible covering by cubes, we prove that
\[
    \|f\|_{H_L^1}
    \asymp
    \|f\|_{L^1(dw)}
    +\sum_{j=1}^N
    \|\widetilde{R}_j f\|_{L^1(dw)}.
\]
The proof combines an atomic characterization of $H_L^1$
with local Hardy space theory and weighted integral
estimates for Dunkl derivatives of the heat kernels.
The result applies to arbitrary root systems and
nonnegative multiplicity functions.
\end{abstract}

\maketitle
\section{Introduction}

Let $R$ be a root system in $\mathbb{R}^N$, let $k\geq0$
be a multiplicity function, and let $dw$ denote the
associated measure. We consider the Dunkl--Schr\"odinger
operator
\[
    L=-\Delta_k+V,
    \qquad V\geq0,\quad V\in L^2_{\mathrm{loc}}(dw),
\]
where $\Delta_k=\sum_{j=1}^N T_j^2$ is the Dunkl Laplacian.
The Hardy space $H_L^1$ associated with $L$ is defined
by the maximal function of the semigroup $K_t=e^{-tL}$:
\[
    H_L^1
    =
    \left\{
        f\in L^1(dw):
        \sup_{t>0}|K_tf|\in L^1(dw)
    \right\},
    \qquad
    \|f\|_{H_L^1}
    =
    \left\|\sup_{t>0}|K_tf|\right\|_{L^1(dw)}.
\]

Riesz transform characterizations are a central part
of the real-variable theory of Hardy spaces. The characterization of Hardy spaces by Riesz transforms
is a classical result of Fefferman and Stein~\cite{FS}.
In this paper, we study an analogous characterization
for Hardy spaces associated with Dunkl--Schr\"odinger operators.
For classical Schr\"odinger operators, such
characterizations were studied in
\cite{DZ_Revista,DzP}.
Hardy spaces in the rational Dunkl setting were studied
in~\cite{Anker15} for the reflection group $\mathbb Z_2^N$.
The theory for general root systems, including a characterization
by Riesz transforms, was developed in~\cite{ADzH},
and a local theory, together with a Riesz transform
characterization for the Dunkl harmonic oscillator,
was established in~\cite{Hejna}.
For more general nonnegative potentials,
\cite{main} provides an atomic characterization of
$H_L^1$ under conditions on the potential and an
associated covering of $\mathbb{R}^N$ by cubes. Related Riesz transform characterizations under
assumptions on the heat semigroup were obtained
in~\cite{KaniaPreisner}, with applications to
multidimensional Bessel and Laguerre operators.

The purpose of this paper is to prove a Riesz transform
characterization under the same assumptions.
Writing $k_t(\mathbf{x},\mathbf{y})$ for the kernel
of $K_t$, we define
\[
    \widetilde{R}_jf
    =
    \lim_{\varepsilon\to0^+}
    \int_{\mathbb{R}^N}
    \int_\varepsilon^{1/\varepsilon}
    T_{j,\mathbf{x}}k_t(\mathbf{x},\mathbf{y})
    \,\frac{dt}{\sqrt{t}}\,
    f(\mathbf{y})\,dw(\mathbf{y}),
\]
where the derivative is understood weakly and the
limit is taken in $\mathcal{S}'(\mathbb{R}^N)$.
We prove that this limit exists for every $f\in L^1(dw)$. The $L^p(dw)$-boundedness of the Dunkl Riesz transforms
for $1<p<\infty$ was established by Amri and
Sifi~\cite{AmriSifi}. For Dunkl--Schr\"odinger operators, Amri and
Hammi~\cite{AH_2} proved weak type $(1,1)$ estimates
and $L^p(dw)$-boundedness of the associated Riesz
transforms for $1<p\leq2$.
Our main result, Theorem~\ref{theo:main}, states that
under conditions~\eqref{eq:finite_overlap},
\eqref{eq:D}, and~\eqref{eq:K},
\[
    f\in H_L^1
    \quad\Longleftrightarrow\quad
    f\in L^1(dw)
    \ \text{and}\
    \widetilde{R}_jf\in L^1(dw)
    \quad(j=1,\ldots,N),
\]
with
\[
    \|f\|_{H_L^1}
    \asymp
    \|f\|_{L^1(dw)}
    +\sum_{j=1}^N\|\widetilde{R}_jf\|_{L^1(dw)}.
\]
Here membership of $\widetilde{R}_jf$ in $L^1(dw)$
means that the distribution is represented by an
$L^1(dw)$ function.

The proof follows the localization approach
of~\cite{DzP}, using the atomic characterization
from~\cite{main} and the local Riesz transforms
from~\cite{Hejna}.
The reflection terms in the Dunkl operators require
estimates that retain information about Euclidean
distance, rather than only distance between
reflection-group orbits.
We obtain weighted integral estimates for weak
Dunkl derivatives of the Schr\"odinger kernel
from the perturbation formula and heat-kernel bounds.
These estimates control the local comparison
operators and commutators needed in the proof.

The paper begins with the required preliminaries
and the distributional construction of the Riesz
transforms. We then state the main result, establish
kernel estimates, and recall the local Hardy-space
theory. The final sections contain the auxiliary
localization lemmas and the proof of the main theorem.

\section{Preliminaries}\label{sec:preliminaries}

\subsection{Dunkl operators and the associated measure}

The foundations of Dunkl analysis were developed
in~\cite{Dunkl0,Dunkl,Dunkl2,Dunkl3}.
We recall below the definitions and estimates needed in this paper.
For further background, see~\cite{Roesler3,Roesler-Voit}.
Throughout the paper, $C$ denotes a positive constant
whose value may change from line to line.
Unless otherwise stated, constants depend only on the
Dunkl structure and, when applicable, on the constants
in conditions~\eqref{eq:finite_overlap},
\eqref{eq:D}, and~\eqref{eq:K}.

Let $R\subset\mathbb{R}^N\setminus\{0\}$ be a finite
normalized root system:
\[
    R\cap\mathbb{R}\alpha=\{\alpha,-\alpha\},
    \qquad
    \sigma_\alpha(R)=R,
    \qquad
    \|\alpha\|=\sqrt{2},
    \qquad \alpha\in R,
\]
where
\[
    \sigma_\alpha(\mathbf{x})
    =
    \mathbf{x}
    -2\frac{\langle\mathbf{x},\alpha\rangle}
            {\|\alpha\|^2}\alpha.
\]
Let $G$ be the finite reflection group generated by
$\{\sigma_\alpha:\alpha\in R\}$, and let
$k:R\to[0,\infty)$ be a fixed $G$-invariant
multiplicity function.
The associated measure is
\begin{equation}\label{eq:measure_formula}
    dw(\mathbf{x})
    =
    \prod_{\alpha\in R}
    |\langle\mathbf{x},\alpha\rangle|^{k(\alpha)}
    \,d\mathbf{x},
\end{equation}
and its homogeneous dimension is
\begin{equation}\label{eq:homo}
    \mathbf{N}=N+\sum_{\alpha\in R}k(\alpha).
\end{equation}
We write $B(\mathbf{x},r)$ for the Euclidean closed ball
with center $\mathbf{x}$ and radius $r$.
The measure $dw$ is $G$-invariant and satisfies
\[
    w(B(t\mathbf{x},tr))
    =t^{\mathbf{N}}w(B(\mathbf{x},r)),
    \qquad t,r>0.
\]
Moreover,
\begin{equation}\label{eq:balls_asymp}
    w(B(\mathbf{x},r))
    \asymp
    r^N\prod_{\alpha\in R}
    \bigl(|\langle\mathbf{x},\alpha\rangle|+r\bigr)^{k(\alpha)}.
\end{equation}
In particular, $dw$ is doubling:
\begin{equation}\label{eq:doubling}
    w(B(\mathbf{x},2r))
    \leq Cw(B(\mathbf{x},r)).
\end{equation}
More generally, for $r_2\geq r_1>0$,
\begin{equation}\label{eq:growth}
    C^{-1}\left(\frac{r_2}{r_1}\right)^N
    \leq
    \frac{w(B(\mathbf{x},r_2))}
         {w(B(\mathbf{x},r_1))}
    \leq
    C\left(\frac{r_2}{r_1}\right)^{\mathbf{N}}.
\end{equation}
We shall also use the uniform lower bound
\begin{equation}\label{eq:volume_lower}
    w(B(\mathbf{x},r))\geq c r^{\mathbf{N}},
    \qquad \mathbf{x}\in\mathbb{R}^N,\quad r>0,
\end{equation}
which follows directly from~\eqref{eq:balls_asymp}.

For $\xi\in\mathbb{R}^N$, the Dunkl operator $T_\xi$
is defined by
\begin{equation}\label{eq:T_def}
    T_\xi f(\mathbf{x})
    =
    \partial_\xi f(\mathbf{x})
    +
    \sum_{\alpha\in R}\frac{k(\alpha)}{2}
    \langle\alpha,\xi\rangle
    \frac{f(\mathbf{x})-f(\sigma_\alpha\mathbf{x})}
         {\langle\alpha,\mathbf{x}\rangle}.
\end{equation}
For $C^1$ functions, the difference quotients are
extended continuously across the reflection hyperplanes.
Let $e_1,\ldots,e_N$ be the standard basis of
$\mathbb{R}^N$, and write $T_j=T_{e_j}$.

The Dunkl operators commute and are skew-symmetric
with respect to $dw$. In particular,
\begin{equation}\label{eq:Dunkl_integration_by_parts}
    \int_{\mathbb{R}^N}T_jf\,g\,dw
    =
    -\int_{\mathbb{R}^N}f\,T_jg\,dw,
    \qquad f,g\in C_c^\infty(\mathbb{R}^N).
\end{equation}
They also preserve $\mathcal{S}(\mathbb{R}^N)$.

For $f,g\in L^2(dw)$, we write $T_jf=g$ in the weak
sense if
\[
    \int_{\mathbb{R}^N}g\varphi\,dw
    =
    -\int_{\mathbb{R}^N}fT_j\varphi\,dw,
    \qquad \varphi\in C_c^\infty(\mathbb{R}^N).
\]

We shall use the following elementary estimate.

\begin{lemma}\label{lem:Dunkl_gradient_Linfty}
For every $f\in C^1(\mathbb{R}^N)$ with bounded gradient
and every $j\in\{1,\ldots,N\}$,
\[
    \|T_jf\|_{L^\infty}
    \leq
    \left(
        1+\sum_{\alpha\in R}
        k(\alpha)\frac{|\alpha_j|}{\|\alpha\|}
    \right)
    \|\nabla f\|_{L^\infty}.
\]
\end{lemma}

\begin{proof}
The fundamental theorem of calculus gives
\[
\begin{aligned}
    |f(\mathbf{x})-f(\sigma_\alpha\mathbf{x})|
    &\leq
    \|\mathbf{x}-\sigma_\alpha\mathbf{x}\|
    \|\nabla f\|_{L^\infty}\\
    &=
    \frac{2|\langle\alpha,\mathbf{x}\rangle|}{\|\alpha\|}
    \|\nabla f\|_{L^\infty}.
\end{aligned}
\]
The assertion follows from~\eqref{eq:T_def}.
\end{proof}

\subsection{The Dunkl heat semigroup}

The Dunkl Laplacian is
\[
    \Delta_k=\sum_{j=1}^N T_j^2.
\]
Its realization on $C_c^\infty(\mathbb{R}^N)$ is
essentially self-adjoint on $L^2(dw)$;
see~\cite[Theorem~3.1]{AH}.
We use the same symbol for its self-adjoint closure.
The semigroup $H_t=e^{t\Delta_k}$ has the representation
\[
    H_tf(\mathbf{x})
    =
    \int_{\mathbb{R}^N}
    h_t(\mathbf{x},\mathbf{y})f(\mathbf{y})\,dw(\mathbf{y}).
\]
The Dunkl heat semigroup and its kernel were studied
by R\"osler~\cite{Roesler2}.
The kernel $h_t$ is smooth for $t>0$ and satisfies
\[
    h_t(\mathbf{x},\mathbf{y})
    =h_t(\mathbf{y},\mathbf{x})>0,
\]
as well as
\begin{equation}\label{eq:h_integral_1}
    \int_{\mathbb{R}^N}
    h_t(\mathbf{x},\mathbf{y})\,dw(\mathbf{x})
    =1.
\end{equation}
Consequently, $H_t$ is a positivity-preserving
contraction on $L^p(dw)$ for $1\leq p\leq\infty$.

The Dunkl derivative of the heat kernel is given by
\begin{equation}\label{eq:heat_derivative}
    T_{j,\mathbf{x}}h_t(\mathbf{x},\mathbf{y})
    =
    \frac{y_j-x_j}{2t}h_t(\mathbf{x},\mathbf{y});
\end{equation}
see~\cite[Lemma~3.3]{DzH1}.

Let
\[
    d(\mathbf{x},\mathbf{y})
    =
    \min_{\sigma\in G}\|\mathbf{x}-\sigma\mathbf{y}\|
\]
be the distance between the $G$-orbits of $\mathbf{x}$
and $\mathbf{y}$, and set
\begin{equation}\label{eq:mathcal_G}
    \mathcal{G}_t(\mathbf{x},\mathbf{y})
    =
    \frac{\exp(-d(\mathbf{x},\mathbf{y})^2/t)}
    {\max\{w(B(\mathbf{x},\sqrt{t})),
           w(B(\mathbf{y},\sqrt{t}))\}}.
\end{equation}

The following estimate follows from
\cite[Theorem~4.1]{ADzH}
and~\cite[Theorem~3.1]{DzH1}. For more precise two-sided estimates of the Dunkl
heat kernel when the multiplicity function is
strictly positive, see~\cite{DzH_heat}.

\begin{theorem}\label{theorem:heat}
There are constants $C,c>0$ such that
\begin{equation}\label{eq:Gauss}
    h_t(\mathbf{x},\mathbf{y})
    \leq
    C\left(
        1+\frac{\|\mathbf{x}-\mathbf{y}\|}{\sqrt{t}}
    \right)^{-2}
    \mathcal{G}_{t/c}(\mathbf{x},\mathbf{y})
\end{equation}
for every $t>0$ and
$\mathbf{x},\mathbf{y}\in\mathbb{R}^N$.
\end{theorem}

We shall use the uniform integrability estimate
\begin{equation}\label{eq:G_integral}
    \sup_{\mathbf{y}\in\mathbb{R}^N}
    \int_{\mathbb{R}^N}
    \mathcal{G}_t(\mathbf{x},\mathbf{y})\,dw(\mathbf{x})
    \leq C,
    \qquad t>0.
\end{equation}

\subsection{The Dunkl--Schr\"odinger operator}
\label{sec:Schorodinger}

Throughout the paper, $V\geq0$ is measurable and
$V\in L^2_{\mathrm{loc}}(dw)$.
We define the Dunkl--Schr\"odinger operator through
the quadratic form
\[
    \mathbf{Q}(f,g)
    =
    \sum_{j=1}^N
    \int_{\mathbb{R}^N}T_jf\,\overline{T_jg}\,dw
    +
    \int_{\mathbb{R}^N}Vf\overline{g}\,dw
\]
with domain
\[
    \mathcal{D}(\mathbf{Q})
    =
    \left\{
        f\in L^2(dw):
        T_jf\in L^2(dw)\ \text{for }j=1,\ldots,N,
        \quad V^{1/2}f\in L^2(dw)
    \right\},
\]
where the derivatives are understood in the weak sense.

The form $\mathbf{Q}$ is densely defined, nonnegative,
and closed; see~\cite[Lemma~4.1]{AH}.
Let $L$ be its associated nonnegative self-adjoint
operator. It is the self-adjoint realization of
\begin{equation}\label{eq:Schrodinger_operator}
    L=-\Delta_k+V.
\end{equation}
More precisely, it is the closure of $-\Delta_k+V$
initially defined on $C_c^\infty(\mathbb{R}^N)$;
see~\cite[Theorem~4.6]{AH}.
We have
\[
    \mathcal{D}(L^{1/2})=\mathcal{D}(\mathbf{Q}),
    \qquad
    \mathbf{Q}(f,f)=\|L^{1/2}f\|_{L^2(dw)}^2,
\]
where $L^{1/2}$ denotes the nonnegative self-adjoint
square root of $L$.

The semigroup $K_t=e^{-tL}$ has an integral kernel:
\[
    K_tf(\mathbf{x})
    =
    \int_{\mathbb{R}^N}
    k_t(\mathbf{x},\mathbf{y})f(\mathbf{y})\,dw(\mathbf{y}).
\]
We use a jointly measurable version of $k_t$.
For every $t>0$,
\[
    k_t(\mathbf{x},\mathbf{y})
    =k_t(\mathbf{y},\mathbf{x}),
\]
and
\begin{equation}\label{eq:kernels_compare}
    0\leq k_t(\mathbf{x},\mathbf{y})
    \leq h_t(\mathbf{x},\mathbf{y})
\end{equation}
for $dw\otimes dw$-almost every
$(\mathbf{x},\mathbf{y})$;
see~\cite[Theorem~4.8]{AH}.
In particular,
\begin{equation}\label{eq:kernel_mass}
    \int_{\mathbb{R}^N}
    k_t(\mathbf{x},\mathbf{y})\,dw(\mathbf{x})
    \leq1
\end{equation}
for almost every $\mathbf{y}$.
Thus $K_t$ extends to a positivity-preserving
contraction on $L^p(dw)$, $1\leq p\leq\infty$.

The semigroup property gives
\begin{equation}\label{eq:kernel_semigroup}
    k_{t+s}(\mathbf{x},\mathbf{y})
    =
    \int_{\mathbb{R}^N}
    k_t(\mathbf{x},\mathbf{z})
    k_s(\mathbf{z},\mathbf{y})\,dw(\mathbf{z}),
    \qquad t,s>0,
\end{equation}
for almost every $(\mathbf{x},\mathbf{y})$.

We shall use the following energy estimate.

\begin{lemma}[\cite{main}, Lemma~18]
\label{lem:L_on_k_t}
For every $t>0$ and almost every $\mathbf{y}$,
the function $k_t(\cdot,\mathbf{y})$ belongs to
$\mathcal{D}(L)$, and
\begin{equation}\label{eq:analytic}
\begin{aligned}
    \mathbf{Q}(k_t(\cdot,\mathbf{y}),k_t(\cdot,\mathbf{y}))
    &=
    \langle Lk_t(\cdot,\mathbf{y}),
            k_t(\cdot,\mathbf{y})\rangle_{L^2(dw)}\leq
    \frac{C}{t\,w(B(\mathbf{y},\sqrt{t}))}.
\end{aligned}
\end{equation}
The constant $C$ is independent of $t$, $\mathbf{y}$,
and $V$.
\end{lemma}

In particular,
\begin{equation}\label{eq:kernel_derivative_L2}
    \|T_{j,\mathbf{x}}k_t(\cdot,\mathbf{y})\|_{L^2(dw)}
    \leq
    \frac{C}
    {\sqrt{t}\,w(B(\mathbf{y},\sqrt{t}))^{1/2}},
    \qquad j=1,\ldots,N.
\end{equation}

Finally, we recall the perturbation formula
(see~\cite[proof of Lemma~16]{main}):
\begin{equation}\label{eq:perturbation_formula}
    h_t(\mathbf{x},\mathbf{y})-k_t(\mathbf{x},\mathbf{y})
    =
    \int_0^t\int_{\mathbb{R}^N}
    h_{t-s}(\mathbf{x},\mathbf{z})
    V(\mathbf{z})k_s(\mathbf{z},\mathbf{y})
    \,dw(\mathbf{z})\,ds.
\end{equation}
This identity holds for every $t>0$ and almost every
$(\mathbf{x},\mathbf{y})$.
The integrand is nonnegative, and the integral is
finite almost everywhere by~\eqref{eq:kernels_compare}. Lower bounds for Dunkl--Schr\"odinger semigroup
kernels with Green bounded potentials were studied
in~\cite{DzH_Green}.

\section{Riesz transforms}

Let $\mathcal{Q}$ be a collection of closed cubes in $\mathbb{R}^N$
with sides parallel to the coordinate axes and pairwise disjoint
interiors, such that
\[
    \bigcup_{Q\in\mathcal{Q}}Q=\mathbb{R}^N.
\]
All cubes are assumed to have positive, finite side-length.
We denote the side-length of $Q$ by $d(Q)$ and let $Q^*$ be
the concentric cube with side-length $2d(Q)$.
Repeated asterisks denote successive applications of this dilation.
Assume that there exists $C_0\geq 1$ such that
\begin{equation}\label{eq:finite_overlap}\tag{F}
    Q_1^{****}\cap Q_2^{****}\neq\varnothing
    \quad\Longrightarrow\quad
    C_0^{-1}d(Q_1)\leq d(Q_2)\leq C_0d(Q_1)
    \qquad(Q_1,Q_2\in\mathcal{Q}).
\end{equation}
The family $\mathcal Q$ is countable, since its members
have pairwise disjoint nonempty interiors.
The family $\{Q^{****}:Q\in\mathcal{Q}\}$
has uniformly bounded overlap.

For $f\in L^1(dw)$, $0<\varepsilon<1$, and $j=1,\ldots,N$,
set
\begin{equation}\label{eq:Rj_truncated}
    \widetilde{R}_j^\varepsilon f(\mathbf{x})
    =
    \int_{\mathbb{R}^N}
    \left(
    \int_\varepsilon^{1/\varepsilon}
    T_{j,\mathbf{x}}k_t(\mathbf{x},\mathbf{y})
    \,\frac{dt}{\sqrt{t}}
    \right)
    f(\mathbf{y})\,dw(\mathbf{y}).
\end{equation}
Here $T_{j,\mathbf{x}}k_t(\cdot,\mathbf{y})$ denotes the weak
Dunkl derivative with respect to the first variable, characterized by
\[
    \int_{\mathbb{R}^N}
    T_{j,\mathbf{x}}k_t(\mathbf{x},\mathbf{y})
    \varphi(\mathbf{x})\,dw(\mathbf{x})
    =
    -\int_{\mathbb{R}^N}
    k_t(\mathbf{x},\mathbf{y})T_j\varphi(\mathbf{x})\,dw(\mathbf{x}),
    \qquad \varphi\in C_c^\infty(\mathbb{R}^N).
\]
Lemma~\ref{lem:L_on_k_t} gives
\[
    \|T_{j,\mathbf{x}}k_t(\cdot,\mathbf{y})\|_{L^2(dw)}
    \leq
    \frac{C}{\sqrt{t}\,w(B(\mathbf{y},\sqrt{t}))^{1/2}}.
\]
By~\eqref{eq:balls_asymp},
$w(B(\mathbf{y},\sqrt{t}))\geq c\,t^{\mathbf{N}/2}$,
uniformly in $\mathbf{y}\in\mathbb{R}^N$ and $t>0$.
Consequently, Minkowski's integral inequality yields
\[
\begin{aligned}
    &\left\|
    \int_{\mathbb{R}^N}|f(\mathbf{y})|
    \int_{\varepsilon}^{1/\varepsilon}
    |T_{j,\mathbf{x}}k_t(\cdot,\mathbf{y})|
    \,\frac{dt}{\sqrt{t}}\,dw(\mathbf{y})
    \right\|_{L^2(dw)}
    \\
    &\qquad\leq
    \int_{\mathbb{R}^N}|f(\mathbf{y})|
    \int_{\varepsilon}^{1/\varepsilon}
    \|T_{j,\mathbf{x}}k_t(\cdot,\mathbf{y})\|_{L^2(dw)}
    \,\frac{dt}{\sqrt{t}}\,dw(\mathbf{y})
    \\
    &\qquad\leq
    C\|f\|_{L^1(dw)}
    \int_{\varepsilon}^{1/\varepsilon}
    t^{-1-\mathbf{N}/4}\,dt
    \leq C_{\varepsilon}\|f\|_{L^1(dw)}.
\end{aligned}
\]
Thus the nonnegative function inside the $L^2(dw)$ norm is
finite almost everywhere. In particular, the integral
in~\eqref{eq:Rj_truncated} is absolutely convergent for almost
every $\mathbf{x}$, and
\[
    \|\widetilde{R}_j^\varepsilon f\|_{L^2(dw)}
    \leq C_{\varepsilon}\|f\|_{L^1(dw)}.
\]

For $Q\in\mathcal{Q}$, define
\begin{equation}\label{eq:R_jepsilon_0}
    \widetilde{R}_{j,Q,0}^\varepsilon f(\mathbf{x})
    =
    \int_{\mathbb{R}^N}
    \left(
    \int_\varepsilon^{1/\varepsilon}
    \mathbf{1}_{(0,d(Q)^2)}(t)
    T_{j,\mathbf{x}}k_t(\mathbf{x},\mathbf{y})
    \,\frac{dt}{\sqrt{t}}
    \right)
    f(\mathbf{y})\,dw(\mathbf{y}),
\end{equation}
and
\begin{equation}\label{eq:R_jepsilon_infty}
    \widetilde{R}_{j,Q,\infty}^\varepsilon f(\mathbf{x})
    =
    \int_{\mathbb{R}^N}
    \left(
    \int_\varepsilon^{1/\varepsilon}
    \mathbf{1}_{[d(Q)^2,\infty)}(t)
    T_{j,\mathbf{x}}k_t(\mathbf{x},\mathbf{y})
    \,\frac{dt}{\sqrt{t}}
    \right)
    f(\mathbf{y})\,dw(\mathbf{y}).
\end{equation}
These operators are also bounded from $L^1(dw)$ to $L^2(dw)$
for each fixed $\varepsilon$, and
\[
    \widetilde{R}_j^\varepsilon
    =
    \widetilde{R}_{j,Q,0}^\varepsilon
    +
    \widetilde{R}_{j,Q,\infty}^\varepsilon.
\]
We denote their integral kernels by
$\widetilde{R}_j^\varepsilon(\mathbf{x},\mathbf{y})$,
$\widetilde{R}_{j,Q,0}^\varepsilon(\mathbf{x},\mathbf{y})$,
and
$\widetilde{R}_{j,Q,\infty}^\varepsilon(\mathbf{x},\mathbf{y})$,
respectively.

The following proposition adapts the distributional
construction in~\cite[Section~2]{DzP}
to the Dunkl setting.
We identify a function $g\in L^2(dw)$ with the tempered
distribution
\[
    \langle g,\phi\rangle
    =\int_{\mathbb{R}^N}g(\mathbf{x})\phi(\mathbf{x})\,dw(\mathbf{x}),
    \qquad \phi\in\mathcal{S}(\mathbb{R}^N).
\]

\begin{proposition}\label{prop:Riesz_distributions}
For every $f\in L^1(dw)$, $Q\in\mathcal{Q}$, and
$j\in\{1,\ldots,N\}$, the limits
\begin{equation}\label{eq:limit_R_jepsilon}
    \widetilde{R}_jf
    :=\lim_{\varepsilon\to0^+}\widetilde{R}_j^\varepsilon f,
\end{equation}
\begin{equation}\label{eq:limit_R_jepsilon_0}
    \widetilde{R}_{j,Q,0}f
    :=\lim_{\varepsilon\to0^+}
    \widetilde{R}_{j,Q,0}^\varepsilon f,
\end{equation}
and
\begin{equation}\label{eq:limit_R_jepsilon_infty}
    \widetilde{R}_{j,Q,\infty}f
    :=\lim_{\varepsilon\to0^+}
    \widetilde{R}_{j,Q,\infty}^\varepsilon f
\end{equation}
exist in $\mathcal{S}'(\mathbb{R}^N)$.
Moreover,
\[
    \widetilde{R}_jf
    =
    \widetilde{R}_{j,Q,0}f
    +
    \widetilde{R}_{j,Q,\infty}f.
\]
There is a constant $C>0$, independent of $f$, $Q$, and $j$,
such that, for every $\phi\in\mathcal{S}(\mathbb{R}^N)$,
\begin{equation}\label{eq:distribution}
\begin{aligned}
    &|\langle\widetilde{R}_jf,\phi\rangle|
    +|\langle\widetilde{R}_{j,Q,0}f,\phi\rangle|
    +|\langle\widetilde{R}_{j,Q,\infty}f,\phi\rangle|\leq
    C\|f\|_{L^1(dw)}
    \left(
        \|\phi\|_{L^2(dw)}
        +\|T_j\phi\|_{L^\infty}
    \right).
\end{aligned}
\end{equation}
\end{proposition}

\begin{proof}
Fix $\phi\in\mathcal{S}(\mathbb{R}^N)$ and set
\[
    F_\phi(t,\mathbf{y})
    =
    \int_{\mathbb{R}^N}
    T_{j,\mathbf{x}}k_t(\mathbf{x},\mathbf{y})
    \phi(\mathbf{x})\,dw(\mathbf{x}).
\]
For $0<t\leq1$, the weak integration-by-parts identity gives
\[
    F_\phi(t,\mathbf{y})
    =
    -\int_{\mathbb{R}^N}
    k_t(\mathbf{x},\mathbf{y})
    T_j\phi(\mathbf{x})\,dw(\mathbf{x}).
\]
The identity extends from compactly supported smooth test
functions to Schwartz functions by approximation with smooth
radial cutoffs, using
$k_t(\cdot,\mathbf{y}),
T_{j,\mathbf{x}}k_t(\cdot,\mathbf{y})\in L^2(dw)$.
Since $0\leq k_t\leq h_t$ and
$\int_{\mathbb{R}^N}h_t(\mathbf{x},\mathbf{y})\,dw(\mathbf{x})=1$,
we obtain
\begin{equation}\label{eq:dist_small_t}
    |F_\phi(t,\mathbf{y})|
    \leq \|T_j\phi\|_{L^\infty},
    \qquad 0<t\leq1.
\end{equation}

For $t\geq1$, the Cauchy--Schwarz inequality,
Lemma~\ref{lem:L_on_k_t}, and~\eqref{eq:balls_asymp} yield
\begin{equation}\label{eq:dist_large_t}
\begin{aligned}
    |F_\phi(t,\mathbf{y})|
    &\leq
    \|\phi\|_{L^2(dw)}
    \|T_{j,\mathbf{x}}k_t(\cdot,\mathbf{y})\|_{L^2(dw)}
    \leq
    \frac{C\|\phi\|_{L^2(dw)}}
    {\sqrt{t}\,w(B(\mathbf{y},\sqrt{t}))^{1/2}}
    \leq
    C\|\phi\|_{L^2(dw)}t^{-1/2-\mathbf{N}/4}.
\end{aligned}
\end{equation}
Consequently, uniformly for almost every $\mathbf{y}$,
\begin{equation}\label{eq:dist_goal}
    \int_0^\infty |F_\phi(t,\mathbf{y})|
    \,\frac{dt}{\sqrt{t}}
    \leq
    C\left(
        \|T_j\phi\|_{L^\infty}
        +\|\phi\|_{L^2(dw)}
    \right).
\end{equation}
In particular,
$|f(\mathbf{y})F_\phi(t,\mathbf{y})|$
is integrable with respect to
$dw(\mathbf{y})\,dt/\sqrt{t}$.

For an interval $I\subseteq(0,\infty)$, define
\[
    \langle\mathcal{R}_{j,I}f,\phi\rangle
    =
    \int_{\mathbb{R}^N}f(\mathbf{y})
    \int_I F_\phi(t,\mathbf{y})
    \,\frac{dt}{\sqrt{t}}\,dw(\mathbf{y}).
\]
By~\eqref{eq:dist_goal},
\[
    |\langle\mathcal{R}_{j,I}f,\phi\rangle|
    \leq
    C\|f\|_{L^1(dw)}
    \left(
        \|T_j\phi\|_{L^\infty}
        +\|\phi\|_{L^2(dw)}
    \right),
\]
with a constant independent of $I$.
The expression on the right is a continuous seminorm of
$\phi$ in $\mathcal{S}(\mathbb{R}^N)$ (see Lemma~\ref{lem:Dunkl_gradient_Linfty}); hence
$\mathcal{R}_{j,I}f$ is a tempered distribution.

For the truncated operators, Fubini's theorem gives the same
pairing with the time interval restricted to
$I\cap[\varepsilon,1/\varepsilon]$.
Dominated convergence therefore proves
\eqref{eq:limit_R_jepsilon}--\eqref{eq:limit_R_jepsilon_infty}
by taking, respectively,
\[
    I=(0,\infty),\qquad
    I=(0,d(Q)^2),\qquad
    I=[d(Q)^2,\infty).
\]
Estimate~\eqref{eq:distribution} follows from
\eqref{eq:dist_goal}, and the asserted decomposition follows
by splitting the time integral at $d(Q)^2$.
\end{proof}

\section{Statement of the results}

Our main result characterizes the Hardy space associated with
$L$ by the Riesz transforms $\widetilde{R}_j$.
We first recall the maximal and atomic definitions and the
conditions on $V$ and $\mathcal{Q}$ under which the corresponding
spaces coincide.

For each $f\in L^1(dw)$, we use a version of $K_tf(\mathbf{x})$
that is continuous in $t>0$ for almost every $\mathbf{x}$.
Such a version exists by the $L^1$--$L^2$ regularization
of the semigroup and its analyticity on $L^2(dw)$.
Consequently, the supremum in the definition below may
equivalently be taken over positive rational times.

\begin{definition}\normalfont
A function $f\in L^1(dw)$ belongs to the Hardy space $H_L^1$
if
\[
    f^*(\mathbf{x})
    =\sup_{t>0}|K_tf(\mathbf{x})|
\]
belongs to $L^1(dw)$. We equip this space with the norm
\[
    \|f\|_{H_L^1}=\|f^*\|_{L^1(dw)}.
\]
\end{definition}

\begin{definition}\label{def:atomic}\normalfont
A measurable function $a$ is called a
$\mathcal{Q}$-atom if there exist $Q\in\mathcal{Q}$
and a ball $B=B(\mathbf{x}_0,r)$ such that
\begin{enumerate}[(A)]
    \item
    $\operatorname{supp}a\subseteq B\subseteq Q^{****}$;
    \label{numitem:supports}

    \item
    $\|a\|_{L^\infty(dw)}\leq w(B)^{-1}$;
    \label{numitem:size}

    \item
    if $r<d(Q)$, then
    $\displaystyle\int_{\mathbb{R}^N}a(\mathbf{x})\,dw(\mathbf{x})=0$.
    \label{numitem:cancellation}
\end{enumerate}
The atomic Hardy space $H_{\mathcal{Q}}^{1,\mathrm{at}}$
consists of all $f\in L^1(dw)$ admitting a representation
\[
    f=\sum_{\nu=1}^{\infty}c_\nu a_\nu
    \quad\text{in }L^1(dw),
    \qquad
    \sum_{\nu=1}^{\infty}|c_\nu|<\infty,
\]
where each $a_\nu$ is a $\mathcal{Q}$-atom.
Its norm is
\[
    \|f\|_{H_{\mathcal{Q}}^{1,\mathrm{at}}}
    =
    \inf\left\{
        \sum_{\nu=1}^{\infty}|c_\nu|:
        f=\sum_{\nu=1}^{\infty}c_\nu a_\nu
        \text{ is such a representation}
    \right\}.
\]
\end{definition}

Every $\mathcal{Q}$-atom satisfies $\|a\|_{L^1(dw)}\leq1$.
Consequently, the series above converges absolutely in $L^1(dw)$
whenever its coefficients are summable, and
\[
    \|f\|_{L^1(dw)}
    \leq \|f\|_{H_{\mathcal{Q}}^{1,\mathrm{at}}}.
\]
With the atomic norm, $H_{\mathcal{Q}}^{1,\mathrm{at}}$
is a Banach space.

Following~\cite{DZ_Studia,main}, we impose two additional
conditions on $V$ and $\mathcal{Q}$.
The first controls the local contribution of the potential:
there exist $C,\delta>0$ such that
\begin{equation}\label{eq:K}\tag{K}
    \int_0^{2t}\int_{Q^{***}}
    V(\mathbf{y})\mathcal{G}_{2s/c}(\mathbf{x},\mathbf{y})
    \,dw(\mathbf{y})\,ds
    \leq C\left(\frac{t}{d(Q)^2}\right)^\delta
\end{equation}
for every $Q\in\mathcal{Q}$, $\mathbf{x}\in\mathbb{R}^N$,
and $0<t\leq d(Q)^2$, where $c>0$ is the constant
in Theorem~\ref{theorem:heat}.
The second requires decay of the total mass of the
Schr\"odinger kernel at large times: there exist $C,\eta>0$
such that
\begin{equation}\label{eq:D}\tag{D}
    \operatorname*{ess\,sup}_{\mathbf{y}\in Q^{****}}
    \int_{\mathbb{R}^N}
    k_{2^n d(Q)^2}(\mathbf{x},\mathbf{y})\,dw(\mathbf{x})
    \leq Cn^{-1-\eta},
    \qquad Q\in\mathcal{Q},\quad n=1,2,\ldots.
\end{equation}

The following atomic characterization was proved
in~\cite[Theorem~7]{main}.

\begin{theorem}\label{theorem:main_prev}
Assume that \eqref{eq:finite_overlap}, \eqref{eq:D},
and \eqref{eq:K} hold.
Then
\[
    H_L^1=H_{\mathcal{Q}}^{1,\mathrm{at}}
\]
with equivalent norms. More precisely, there is a constant
$C>0$ such that
\[
    C^{-1}\|f\|_{H_{\mathcal{Q}}^{1,\mathrm{at}}}
    \leq \|f\|_{H_L^1}
    \leq C\|f\|_{H_{\mathcal{Q}}^{1,\mathrm{at}}}
\]
for every $f$ in either of these spaces.
\end{theorem}

Our main result gives a Riesz transform characterization
of the same Hardy space under the same assumptions.
Here the transforms $\widetilde{R}_jf$ are understood
as the tempered distributions constructed in
Proposition~\ref{prop:Riesz_distributions}.
The statement $\widetilde{R}_jf\in L^1(dw)$ means that
this distribution is represented by a function in $L^1(dw)$,
with respect to the pairing
$\langle g,\phi\rangle=\int g\phi\,dw$.

\begin{theorem}\label{theo:main}
Assume that \eqref{eq:finite_overlap}, \eqref{eq:D},
and \eqref{eq:K} hold.
For $f\in L^1(dw)$, we have $f\in H_L^1$ if and only if
$\widetilde{R}_jf\in L^1(dw)$ for every $j=1,\ldots,N$.
Moreover, there is a constant $C>0$ such that
\[
    C^{-1}\|f\|_{H_L^1}
    \leq
    \|f\|_{L^1(dw)}
    +\sum_{j=1}^N\|\widetilde{R}_jf\|_{L^1(dw)}
    \leq C\|f\|_{H_L^1},
    \qquad f\in H_L^1.
\]
\end{theorem}

\begin{corollary}\label{coro:reverse_Holder}
Let $V\geq0$, $V\not\equiv0$, and
$V\in L^2_{\mathrm{loc}}(dw)\cap\mathrm{RH}^q(dw)$,
where
\[
    q>\max\left\{1,\frac{\mathbf N}{2}\right\}.
\]
Then, for $L=-\Delta_k+V$, the characterization and
the norm equivalence in Theorem~\ref{theo:main} hold.
\end{corollary}

\begin{proof}
By~\cite[Section 10]{main}, the covering associated
with $V$ satisfies~\eqref{eq:finite_overlap},
\eqref{eq:D}, and~\eqref{eq:K}.
The conclusion follows from Theorem~\ref{theo:main}.
\end{proof}

\section{Kernel estimates}

We first recall the following standard estimate,
proved in the Dunkl setting in~\cite[Lemma~16]{main}.

\begin{lemma}\label{lem:VK}
For every $f\in L^1(dw)$,
\begin{equation}\label{eq:VK}
    \int_{\mathbb{R}^N}\int_0^\infty
    V(\mathbf{x})K_s|f|(\mathbf{x})
    \,ds\,dw(\mathbf{x})
    \leq \|f\|_{L^1(dw)}.
\end{equation}
Consequently,
\begin{equation}\label{eq:VK_kernel}
    \int_0^\infty\int_{\mathbb{R}^N}
    V(\mathbf{z})k_s(\mathbf{z},\mathbf{y})
    \,dw(\mathbf{z})\,ds
    \leq1
\end{equation}
for $dw$-almost every $\mathbf{y}\in\mathbb{R}^N$.
\end{lemma}

Indeed, by Tonelli's theorem, the left-hand side
of~\eqref{eq:VK} equals
\[
    \int_{\mathbb{R}^N}|f(\mathbf{y})|
    \left(
        \int_0^\infty\int_{\mathbb{R}^N}
        V(\mathbf{z})k_s(\mathbf{z},\mathbf{y})
        \,dw(\mathbf{z})\,ds
    \right)\,dw(\mathbf{y}).
\]
Since~\eqref{eq:VK} holds for every $f\in L^1(dw)$,
it implies~\eqref{eq:VK_kernel}.

We shall need the following weighted version of this estimate.

\begin{lemma}\label{lem:better_VK}
There is a constant $C>0$, depending only on the Dunkl
structure, such that, for every $T>0$ and $dw$-almost every
$\mathbf{y}\in\mathbb{R}^N$,
\begin{equation}\label{eq:weighted_VK}
    \int_0^T\int_{\mathbb{R}^N}
    \left(1+\frac{\|\mathbf{y}-\mathbf{z}\|}{\sqrt{T}}\right)^2
    V(\mathbf{z})k_s(\mathbf{z},\mathbf{y})
    \,dw(\mathbf{z})\,ds
    \leq C.
\end{equation}
\end{lemma}

\begin{proof}
For brevity, set
\[
    W_T(\mathbf{x},\mathbf{y})
    =
    \left(1+\frac{\|\mathbf{x}-\mathbf{y}\|}{\sqrt{T}}\right)^2.
\]
Theorem~\ref{theorem:heat} and the uniform bound
$\int_{\mathbb{R}^N}\mathcal{G}_{u/c}(\mathbf{x},\mathbf{z})
\,dw(\mathbf{x})\leq C$ imply
\[
    \int_{\mathbb{R}^N}
    W_u(\mathbf{x},\mathbf{z})h_u(\mathbf{x},\mathbf{z})
    \,dw(\mathbf{x})\leq C.
\]
Since $W_T\leq W_u$ whenever $0<u\leq T$, we have
\begin{equation}\label{eq:heat_weighted_moment}
    \int_{\mathbb{R}^N}
    W_T(\mathbf{x},\mathbf{z})h_u(\mathbf{x},\mathbf{z})
    \,dw(\mathbf{x})\leq C,
    \qquad 0<u\leq T.
\end{equation}

Fix $T>0$. By~\eqref{eq:perturbation_formula} and the positivity of $k_T$,
\begin{equation}\label{eq:pert_cons}
    \int_0^T\int_{\mathbb{R}^N}
    h_{T-s}(\mathbf{x},\mathbf{z})
    V(\mathbf{z})k_s(\mathbf{z},\mathbf{y})
    \,dw(\mathbf{z})\,ds
    \leq h_T(\mathbf{x},\mathbf{y}).
\end{equation}
Moreover, the triangle inequality gives
\begin{equation}\label{eq:reverse_triangle}
    W_T(\mathbf{z},\mathbf{y})
    \leq
    2W_T(\mathbf{x},\mathbf{y})
    +2W_T(\mathbf{x},\mathbf{z}).
\end{equation}
Multiply~\eqref{eq:reverse_triangle} by
\[
    h_{T-s}(\mathbf{x},\mathbf{z})
    V(\mathbf{z})k_s(\mathbf{z},\mathbf{y})
\]
and integrate in $\mathbf{x}$, $\mathbf{z}$, and $s\in(0,T)$.
Using Tonelli's theorem, the normalization of $h_{T-s}$,
\eqref{eq:pert_cons}, and~\eqref{eq:heat_weighted_moment},
we obtain
\[
\begin{aligned}
    &\int_0^T\int_{\mathbb{R}^N}
    W_T(\mathbf{z},\mathbf{y})
    V(\mathbf{z})k_s(\mathbf{z},\mathbf{y})
    \,dw(\mathbf{z})\,ds
    \\
    &\quad\leq
    2\int_{\mathbb{R}^N}
    W_T(\mathbf{x},\mathbf{y})h_T(\mathbf{x},\mathbf{y})
    \,dw(\mathbf{x})
    \\
    &\qquad\quad+
    2\int_0^T\int_{\mathbb{R}^N}
    \left(
        \int_{\mathbb{R}^N}
        W_T(\mathbf{x},\mathbf{z})
        h_{T-s}(\mathbf{x},\mathbf{z})\,dw(\mathbf{x})
    \right)
    V(\mathbf{z})k_s(\mathbf{z},\mathbf{y})
    \,dw(\mathbf{z})\,ds
    \\
    &\quad\leq
    C+C\int_0^T\int_{\mathbb{R}^N}
    V(\mathbf{z})k_s(\mathbf{z},\mathbf{y})
    \,dw(\mathbf{z})\,ds
    \leq C,
\end{aligned}
\]
where the last inequality follows from~\eqref{eq:VK_kernel}.
\end{proof}

\begin{lemma}\label{lem:weighted_kernel_derivatives}
There is a constant $C>0$, depending only on the Dunkl
structure, such that, for every $j\in\{1,\ldots,N\}$ and $r>0$,
\begin{equation}\label{eq:weight_heat}
    \int_0^{r^2}\int_{\mathbb{R}^N}
    \frac{\|\mathbf{x}-\mathbf{y}\|}{r}
    |T_{j,\mathbf{x}}h_t(\mathbf{x},\mathbf{y})|
    \,dw(\mathbf{x})\,\frac{dt}{\sqrt{t}}
    \leq C
\end{equation}
for every $\mathbf{y}\in\mathbb{R}^N$, and
\begin{equation}\label{eq:weight_schrodinger}
    \int_0^{r^2}\int_{\mathbb{R}^N}
    \frac{\|\mathbf{x}-\mathbf{y}\|}{r}
    |T_{j,\mathbf{x}}k_t(\mathbf{x},\mathbf{y})|
    \,dw(\mathbf{x})\,\frac{dt}{\sqrt{t}}
    \leq C
\end{equation}
for $dw$-almost every $\mathbf{y}\in\mathbb{R}^N$.
\end{lemma}

\begin{proof}
By~\eqref{eq:heat_derivative}
and~\eqref{eq:heat_weighted_moment}, for every $u>0$,
\begin{equation}\label{eq:heat_derivative_moments}
\begin{aligned}
    \int_{\mathbb{R}^N}
    |T_{j,\mathbf{x}}h_u(\mathbf{x},\mathbf{z})|
    \,dw(\mathbf{x})
    &\leq Cu^{-1/2},\qquad
    \int_{\mathbb{R}^N}
    \|\mathbf{x}-\mathbf{z}\|
    |T_{j,\mathbf{x}}h_u(\mathbf{x},\mathbf{z})|
    \,dw(\mathbf{x})\leq C.
\end{aligned}
\end{equation}
The second estimate immediately gives~\eqref{eq:weight_heat},
since
\[
    \frac{1}{r}\int_0^{r^2}\frac{dt}{\sqrt{t}}=2.
\]

We first justify differentiation of the perturbation formula.
By Tonelli's theorem, \eqref{eq:heat_derivative_moments},
and~\eqref{eq:VK_kernel}, for every $R>0$ and almost every
$\mathbf{y}$,
\[
\begin{aligned}
    &\int_0^R\int_0^t\int_{\mathbb{R}^N}
    \|T_{j,\mathbf{x}}h_{t-s}(\cdot,\mathbf{z})\|_{L^1(dw)}
    V(\mathbf{z})k_s(\mathbf{z},\mathbf{y})
    \,dw(\mathbf{z})\,ds\,dt\\
    &\qquad\leq
    C\sqrt{R}\int_0^R\int_{\mathbb{R}^N}
    V(\mathbf{z})k_s(\mathbf{z},\mathbf{y})
    \,dw(\mathbf{z})\,ds
    \leq C\sqrt{R}.
\end{aligned}
\]
Thus, testing~\eqref{eq:perturbation_formula} against
$\varphi\in C_c^\infty(\mathbb{R}^N)$ and using Fubini's
theorem and weak integration by parts, we obtain
\begin{equation}\label{eq:perturbation_derivative}
\begin{aligned}
    T_{j,\mathbf{x}}(h_t-k_t)(\mathbf{x},\mathbf{y})
    =
    \int_0^t\int_{\mathbb{R}^N}
    T_{j,\mathbf{x}}h_{t-s}(\mathbf{x},\mathbf{z})
    V(\mathbf{z})k_s(\mathbf{z},\mathbf{y})
    \,dw(\mathbf{z})\,ds
\end{aligned}
\end{equation}
for almost every $(t,\mathbf{x},\mathbf{y})$.

Fix $r>0$. Set
\[
    A(s,\mathbf{y})
    =\int_{\mathbb{R}^N}
    V(\mathbf{z})k_s(\mathbf{z},\mathbf{y})\,dw(\mathbf{z}),
    \qquad
    B(s,\mathbf{y})
    =\int_{\mathbb{R}^N}
    \|\mathbf{z}-\mathbf{y}\|
    V(\mathbf{z})k_s(\mathbf{z},\mathbf{y})\,dw(\mathbf{z}).
\]
By~\eqref{eq:VK_kernel} and Lemma~\ref{lem:better_VK},
\begin{equation}\label{eq:VK_first_moment}
    \int_0^\infty A(s,\mathbf{y})\,ds\leq1,
    \qquad
    \int_0^T B(s,\mathbf{y})\,ds\leq C\sqrt{T}.
\end{equation}
The second estimate holds for each $T>0$ and almost every
$\mathbf{y}$; by Fubini's theorem it may also be used for
almost every $T$ with $\mathbf{y}$ fixed outside a null set.

By~\eqref{eq:perturbation_derivative}, the triangle inequality,
and~\eqref{eq:heat_derivative_moments},
\[
\begin{aligned}
    &\int_0^{r^2}\int_{\mathbb{R}^N}
    \frac{\|\mathbf{x}-\mathbf{y}\|}{r}
    |T_{j,\mathbf{x}}(h_t-k_t)(\mathbf{x},\mathbf{y})|
    \,dw(\mathbf{x})\,\frac{dt}{\sqrt{t}}\\
    &\qquad\leq
    \frac{C}{r}\int_0^{r^2}\int_0^t
    \left(
        A(s,\mathbf{y})
        +\frac{B(s,\mathbf{y})}{\sqrt{t-s}}
    \right)\,ds\,\frac{dt}{\sqrt{t}}.
\end{aligned}
\]
The contribution of $A$ is bounded by
\[
    \frac{C}{r}\int_0^{r^2}\frac{dt}{\sqrt{t}}\leq C.
\]
For the contribution of $B$ with $0<s<t/2$, we use
$t-s\geq t/2$ and~\eqref{eq:VK_first_moment} to obtain
\[
    \frac{C}{r}\int_0^{r^2}
    \left(\int_0^{t/2}B(s,\mathbf{y})\,ds\right)
    \frac{dt}{t}
    \leq
    \frac{C}{r}\int_0^{r^2}\frac{dt}{\sqrt{t}}
    \leq C.
\]
For the remaining contribution, Tonelli's theorem gives
\[
\begin{aligned}
    &\frac{C}{r}\int_0^{r^2}\int_{t/2}^t
    \frac{B(s,\mathbf{y})}{\sqrt{t-s}}
    \,ds\,\frac{dt}{\sqrt{t}}\\
    &\qquad=
    \frac{C}{r}\int_0^{r^2}B(s,\mathbf{y})
    \left(
        \int_s^{\min(r^2,2s)}
        \frac{dt}{\sqrt{t-s}\sqrt{t}}
    \right)\,ds
    \leq C,
\end{aligned}
\]
where we used
\begin{equation}\label{eq:square_root_integral}
    \int_s^{\min(r^2,2s)}
    \frac{dt}{\sqrt{t-s}\sqrt{t}}
    \leq
    \int_1^2\frac{du}{\sqrt{u-1}\sqrt{u}}
    <\infty,
    \qquad 0<s<r^2,
\end{equation}
and~\eqref{eq:VK_first_moment} with $T=r^2$.
Combining these estimates with~\eqref{eq:weight_heat}
proves~\eqref{eq:weight_schrodinger}.
\end{proof}

\begin{corollary}\label{coro:outside}
There is a constant $C>0$ such that, for every
$Q\in\mathcal{Q}$ and $j\in\{1,\ldots,N\}$,
\[
    \int_0^{d(Q)^2}
    \int_{\mathbb{R}^N\setminus Q^{**}}
    |T_{j,\mathbf{x}}h_t(\mathbf{x},\mathbf{y})|
    \,dw(\mathbf{x})\,\frac{dt}{\sqrt{t}}
    \leq C
\]
for every $\mathbf{y}\in Q^*$, and
\[
    \int_0^{d(Q)^2}
    \int_{\mathbb{R}^N\setminus Q^{**}}
    |T_{j,\mathbf{x}}k_t(\mathbf{x},\mathbf{y})|
    \,dw(\mathbf{x})\,\frac{dt}{\sqrt{t}}
    \leq C
\]
for $dw$-almost every $\mathbf{y}\in Q^*$.
\end{corollary}

\begin{proof}
If $\mathbf{y}\in Q^*$ and $\mathbf{x}\notin Q^{**}$,
then $\|\mathbf{x}-\mathbf{y}\|>d(Q)$. Hence
\[
    \mathbf{1}_{\mathbb{R}^N\setminus Q^{**}}(\mathbf{x})
    \leq \frac{\|\mathbf{x}-\mathbf{y}\|}{d(Q)}.
\]
The assertions follow from~\eqref{eq:weight_heat}
and~\eqref{eq:weight_schrodinger} with $r=d(Q)$.
\end{proof}

\begin{corollary}\label{coro:logarithm}
There is a constant $C>0$ such that, for every $r>0$,
$j\in\{1,\ldots,N\}$, and $dw$-almost every
$\mathbf{y}\in\mathbb{R}^N$,
\begin{equation}\label{eq:logarithm}
\begin{aligned}
    &\int_{r^2}^{2r^2}\int_{\mathbb{R}^N}
    |T_{j,\mathbf{x}}k_{t-r^2/2}(\mathbf{x},\mathbf{y})|
    \,dw(\mathbf{x})\,\frac{dt}{\sqrt{t}}\leq
    \int_{r^2/2}^{3r^2/2}\int_{\mathbb{R}^N}
    |T_{j,\mathbf{x}}k_t(\mathbf{x},\mathbf{y})|
    \,dw(\mathbf{x})\,\frac{dt}{\sqrt{t}}
    \leq C.
\end{aligned}
\end{equation}
\end{corollary}

\begin{proof}
The first inequality follows by the substitution
$s=t-r^2/2$ and the inequality
$(s+r^2/2)^{-1/2}\leq s^{-1/2}$.

To prove the second inequality, split the spatial integral
into the regions $\|\mathbf{x}-\mathbf{y}\|>\sqrt{t}$
and $\|\mathbf{x}-\mathbf{y}\|\leq\sqrt{t}$, and denote
the corresponding contributions by $I_1$ and $I_2$.

Set $R=\sqrt{3/2}\,r$. Since $t\geq r^2/2$ on the
interval of integration, \eqref{eq:weight_schrodinger} gives
\[
\begin{aligned}
    I_1
    &\leq
    \int_{r^2/2}^{3r^2/2}\int_{\mathbb{R}^N}
    \frac{\|\mathbf{x}-\mathbf{y}\|}{\sqrt{t}}
    |T_{j,\mathbf{x}}k_t(\mathbf{x},\mathbf{y})|
    \,dw(\mathbf{x})\,\frac{dt}{\sqrt{t}}
    \\
    &\leq
    \sqrt{3}\int_0^{R^2}\int_{\mathbb{R}^N}
    \frac{\|\mathbf{x}-\mathbf{y}\|}{R}
    |T_{j,\mathbf{x}}k_t(\mathbf{x},\mathbf{y})|
    \,dw(\mathbf{x})\,\frac{dt}{\sqrt{t}}
    \leq C.
\end{aligned}
\]
For $I_2$, we apply the Cauchy--Schwarz inequality
on $B(\mathbf{y},\sqrt{t})$:
\[
\begin{aligned}
    &\int_{B(\mathbf{y},\sqrt{t})}
    |T_{j,\mathbf{x}}k_t(\mathbf{x},\mathbf{y})|
    \,dw(\mathbf{x})
    \\
    &\qquad\leq
    \left(\int_{B(\mathbf{y},\sqrt{t})}1\,dw(\mathbf{x})\right)^{1/2}
    \left(\int_{B(\mathbf{y},\sqrt{t})}
    |T_{j,\mathbf{x}}k_t(\mathbf{x},\mathbf{y})|^2
    \,dw(\mathbf{x})\right)^{1/2}
    \\
    &\qquad\leq
    w(B(\mathbf{y},\sqrt{t}))^{1/2}
    \|T_{j,\mathbf{x}}k_t(\cdot,\mathbf{y})\|_{L^2(dw)}.
\end{aligned}
\]
Moreover, by the definition of $\mathbf{Q}$ and
Lemma~\ref{lem:L_on_k_t},
\[
    \|T_{j,\mathbf{x}}k_t(\cdot,\mathbf{y})\|_{L^2(dw)}^2
    \leq
    \mathbf{Q}(k_t(\cdot,\mathbf{y}),k_t(\cdot,\mathbf{y}))
    \leq \frac{C}{t\,w(B(\mathbf{y},\sqrt{t}))}.
\]
Combining these estimates, we obtain
\[
    I_2
    \leq C\int_{r^2/2}^{3r^2/2}
    \frac{1}{\sqrt{t}}\,\frac{dt}{\sqrt{t}}
    =C\log3.
\]
This completes the proof.
\end{proof}

\section{Local Hardy spaces}

We recall the atomic definition of local Hardy spaces
in the Dunkl setting and their Riesz transform characterization
from~\cite[Section 5]{Hejna}, following the classical local
theory of Goldberg~\cite{Goldberg}.
For a more recent treatment of local Hardy spaces and singular
integral operators in the Dunkl setting, see~\cite{HHLT_local}.

\begin{definition}\label{def:atomic_Dunkl}\normalfont
Let $T>0$. A measurable function $a$ is called an atom
for $H^1_{\mathrm{loc},T}$ if there exists a ball
$B=B(\mathbf{x}_0,r)$, with $r>0$, such that
\begin{enumerate}[(A)]
    \item $\operatorname{supp}a\subseteq B$;
    \label{numitem:support}

    \item $\|a\|_{L^\infty(dw)}\leq w(B)^{-1}$;
    \label{numitem:L_infty}

    \item if $r<T$, then
    $\displaystyle\int_{\mathbb{R}^N}a(\mathbf{x})\,dw(\mathbf{x})=0$.
    \label{numitem:cancellations}
\end{enumerate}
The space $H^1_{\mathrm{loc},T}$ consists of all
$f\in L^1(dw)$ admitting a representation
\begin{equation}\label{eq:atomic_representation_local}
    f=\sum_{\nu=1}^{\infty}c_\nu a_\nu
    \quad\text{in }L^1(dw),
    \qquad
    \sum_{\nu=1}^{\infty}|c_\nu|<\infty,
\end{equation}
where the $a_\nu$ are atoms for $H^1_{\mathrm{loc},T}$.
We set
\[
    \|f\|_{H^1_{\mathrm{loc},T}}
    =
    \inf\left\{\sum_{\nu=1}^{\infty}|c_\nu|\right\},
\]
where the infimum is taken over all
representations~\eqref{eq:atomic_representation_local}.
\end{definition}

For consistency with the normalization of $\widetilde{R}_j$,
we omit the fixed multiplicative constant used
in~\cite{Hejna} when defining the local Riesz transforms.

\begin{definition}\normalfont
For $T>0$, $j\in\{1,\ldots,N\}$, and $f\in L^1(dw)$,
define
\[
    R_j^Tf
    =
    \lim_{\varepsilon\to0^+}
    \int_\varepsilon^{T^2}T_jH_tf\,\frac{dt}{\sqrt{t}}
    \quad\text{in }\mathcal{S}'(\mathbb{R}^N).
\]
Equivalently, for $\phi\in\mathcal{S}(\mathbb{R}^N)$,
\[
    \langle R_j^Tf,\phi\rangle
    =
    -\int_0^{T^2}\int_{\mathbb{R}^N}
    H_tf(\mathbf{x})T_j\phi(\mathbf{x})
    \,dw(\mathbf{x})\,\frac{dt}{\sqrt{t}}.
\]
\end{definition}

The latter integral is absolutely convergent, since
$H_t$ is a contraction on $L^1(dw)$ and
\[
    \int_0^{T^2}\int_{\mathbb{R}^N}
    |H_tf(\mathbf{x})T_j\phi(\mathbf{x})|
    \,dw(\mathbf{x})\,\frac{dt}{\sqrt{t}}
    \leq 2T\|f\|_{L^1(dw)}\|T_j\phi\|_{L^\infty}.
\]
Together with Lemma~\ref{lem:Dunkl_gradient_Linfty},
this also shows that the pairing defines a tempered distribution.

The following characterization follows from the equivalence
of the atomic and maximal definitions of the local Hardy space
and its Riesz transform characterization
in~\cite[Propositions 5.15 and 5.18]{Hejna}.

\begin{proposition}[\cite{Hejna}, Propositions~5.15 and~5.18]
\label{propo:local_Riesz}
Let $T>0$ and $f\in L^1(dw)$.
Then $f\in H^1_{\mathrm{loc},T}$ if and only if
$R_j^Tf\in L^1(dw)$ for every $j=1,\ldots,N$,
where membership in $L^1(dw)$ is understood in the
distributional sense described above.
Moreover, there is a constant $C>0$, independent of $T$
and $f$, such that
\[
    C^{-1}\|f\|_{H^1_{\mathrm{loc},T}}
    \leq
    \|f\|_{L^1(dw)}
    +\sum_{j=1}^N\|R_j^Tf\|_{L^1(dw)}
    \leq C\|f\|_{H^1_{\mathrm{loc},T}},
    \qquad f\in H^1_{\mathrm{loc},T}.
\]
\end{proposition}

\section{Auxiliary lemmas}

The arguments in this section follow the approach
of~\cite[Section~2]{DzP}.
For $Q\in\mathcal{Q}$, set
\begin{equation}\label{eq:Q_prim}
    \mathcal{Q}'(Q)
    =
    \{P\in\mathcal{Q}:Q^{***}\cap P^{***}\neq\varnothing\},
\end{equation}
and
\begin{equation}\label{eq:Q_bis}
    \mathcal{Q}''(Q)
    =
    \{P\in\mathcal{Q}:Q^{***}\cap P^{***}=\varnothing\}.
\end{equation}
Condition~\eqref{eq:finite_overlap} implies that
$d(P)\asymp d(Q)$ for $P\in\mathcal{Q}'(Q)$.
Together with the disjointness of the interiors of the cubes,
this also gives
\[
    \sup_{Q\in\mathcal{Q}}\#\mathcal{Q}'(Q)<\infty.
\]

\begin{lemma}\label{lem:R_j_infty}
Assume that~\eqref{eq:finite_overlap} and~\eqref{eq:D} hold.
There is a constant $C>0$ such that, for every
$Q\in\mathcal{Q}$, $j\in\{1,\ldots,N\}$, and
$dw$-almost every
\[
    \mathbf{y}\in
    \bigcup_{P\in\mathcal{Q}'(Q)}P^{****},
\]
we have
\begin{equation}\label{eq:R_j_infty}
    \int_{\mathbb{R}^N}
    \sup_{0<\varepsilon<1}
    |\widetilde{R}_{j,Q,\infty}^{\varepsilon}
    (\mathbf{x},\mathbf{y})|
    \,dw(\mathbf{x})
    \leq C.
\end{equation}
\end{lemma}

\begin{proof}
By the definition of the truncated kernels,
\[
    \sup_{0<\varepsilon<1}
    |\widetilde{R}_{j,Q,\infty}^{\varepsilon}
    (\mathbf{x},\mathbf{y})|
    \leq
    \int_{d(Q)^2}^{\infty}
    |T_{j,\mathbf{x}}k_t(\mathbf{x},\mathbf{y})|
    \,\frac{dt}{\sqrt{t}}.
\]
It therefore suffices to prove
\begin{equation}\label{eq:R_j_infty_absolute}
    \int_{d(Q)^2}^{\infty}\int_{\mathbb{R}^N}
    |T_{j,\mathbf{x}}k_t(\mathbf{x},\mathbf{y})|
    \,dw(\mathbf{x})\,\frac{dt}{\sqrt{t}}
    \leq C.
\end{equation}

Choose $P\in\mathcal{Q}'(Q)$ with $\mathbf{y}\in P^{****}$.
The left-hand side of~\eqref{eq:R_j_infty_absolute}
is bounded by $S_1+S_2$, where
\[
\begin{aligned}
    S_1
    &=
    \int_{\min\{d(Q)^2,d(P)^2\}}^{d(P)^2}
    \int_{\mathbb{R}^N}
    |T_{j,\mathbf{x}}k_t(\mathbf{x},\mathbf{y})|
    \,dw(\mathbf{x})\,\frac{dt}{\sqrt{t}},\\
    S_2
    &=
    \int_{d(P)^2}^{\infty}\int_{\mathbb{R}^N}
    |T_{j,\mathbf{x}}k_t(\mathbf{x},\mathbf{y})|
    \,dw(\mathbf{x})\,\frac{dt}{\sqrt{t}}.
\end{aligned}
\]
By~\eqref{eq:finite_overlap}, $d(P)\asymp d(Q)$.
Consequently, the interval defining $S_1$ can be covered
by a bounded number of intervals of the form $[a,2a]$,
where the number depends only on $C_0$.
Corollary~\ref{coro:logarithm} bounds the integral over
each such interval uniformly, and hence $S_1\leq C$.

To estimate $S_2$, set $a_n=2^n d(P)^2$.
The semigroup property yields, in the weak sense,
\[
    T_{j,\mathbf{x}}k_t(\mathbf{x},\mathbf{y})
    =
    \int_{\mathbb{R}^N}
    T_{j,\mathbf{x}}k_{t-a_n/2}(\mathbf{x},\mathbf{z})
    k_{a_n/2}(\mathbf{z},\mathbf{y})\,dw(\mathbf{z}),
    \qquad a_n\leq t\leq2a_n.
\]
Indeed, this identity follows by testing the semigroup
formula against a compactly supported smooth function
and using weak integration by parts.
The required interchange of integrals is justified by
Lemma~\ref{lem:L_on_k_t}, the uniform lower bound
$w(B(\mathbf{z},\sqrt{u}))\gtrsim u^{\mathbf{N}/2}$ (see~\eqref{eq:volume_lower}),
and
$\int k_{a_n/2}(\mathbf{z},\mathbf{y})\,dw(\mathbf{z})\leq1$.

Thus, Tonelli's theorem and~\eqref{eq:logarithm},
with $r=\sqrt{a_n}$, give
\[
\begin{aligned}
    S_2
    &\leq
    \sum_{n=0}^{\infty}
    \int_{\mathbb{R}^N}k_{a_n/2}(\mathbf{z},\mathbf{y})
    \left(
        \int_{a_n}^{2a_n}\int_{\mathbb{R}^N}
        |T_{j,\mathbf{x}}k_{t-a_n/2}
        (\mathbf{x},\mathbf{z})|
        \,dw(\mathbf{x})\,\frac{dt}{\sqrt{t}}
    \right)\,dw(\mathbf{z})\\
    &\leq
    C\sum_{n=0}^{\infty}
    \int_{\mathbb{R}^N}
    k_{2^{n-1}d(P)^2}(\mathbf{z},\mathbf{y})
    \,dw(\mathbf{z}).
\end{aligned}
\]
For $n=0,1$, the domination $0\leq k_t\leq h_t$
and the normalization of the heat kernel give
\[
    \sum_{n=0}^{1}
    \int_{\mathbb{R}^N}
    k_{2^{n-1}d(P)^2}(\mathbf{z},\mathbf{y})
    \,dw(\mathbf{z})
    \leq 2.
\]
For $n\geq2$, condition~\eqref{eq:D} applied to $P$
with index $n-1$ gives
\[
    \int_{\mathbb{R}^N}
    k_{2^{n-1}d(P)^2}(\mathbf{z},\mathbf{y})
    \,dw(\mathbf{z})
    \leq C(n-1)^{-1-\eta}.
\]
Therefore,
\[
    S_2\leq C+C\sum_{n=2}^{\infty}(n-1)^{-1-\eta}<\infty,
\]
which proves~\eqref{eq:R_j_infty_absolute}
and hence~\eqref{eq:R_j_infty}.
\end{proof}

For $0<\varepsilon<1$, define
\[
    W_{j,Q}^{\varepsilon}f(\mathbf{x})
    =
    \int_{\mathbb{R}^N}
    \int_{\varepsilon}^{1/\varepsilon}
    \mathbf{1}_{(0,d(Q)^2)}(t)
    T_{j,\mathbf{x}}(h_t-k_t)(\mathbf{x},\mathbf{y})
    \,\frac{dt}{\sqrt{t}}\,
    f(\mathbf{y})\,dw(\mathbf{y}).
\]

For $f\in L^1(dw)$, define
\begin{equation}\label{eq:W_other_formula}
    W_{j,Q}f
    =R_j^{d(Q)}f-\widetilde{R}_{j,Q,0}f
    \quad\text{in }\mathcal{S}'(\mathbb{R}^N).
\end{equation}
Equivalently, $W_{j,Q}f$ is the distributional limit of
$W_{j,Q}^{\varepsilon}f$ as $\varepsilon\to0^+$.

Let $\{\phi_{Q}\}_{Q \in \mathcal{Q}}$ be a smooth resolution of identity associated with $\mathcal{Q}$, that means the collection of $C^{\infty}$-functions on $\mathbb{R}^N$ such that $\supp \phi_{Q} \subseteq Q^{*}$, $0 \leq \phi_Q(\mathbf{x}) \leq 1$,
\begin{equation}\label{eq:partition_derivative}
    |\partial^{\alpha}\phi_{Q}(\mathbf{x})| \leq C_{\alpha}d(Q)^{-|\alpha|} \text{ for all }\alpha \in \mathbb{N}_0^{N},
\end{equation}
and $\sum_{Q \in \mathcal{Q}}\phi_Q(\mathbf{x})=1$ for all $\mathbf{x} \in \mathbb{R}^N$. The existence of $\{\phi_{Q}\}_{Q \in \mathcal{Q}}$ is guaranteed by~\eqref{eq:finite_overlap}.

\begin{lemma}\label{lem:diff}
Assume that~\eqref{eq:K} holds.
There is a constant $C>0$ such that, for every
$Q\in\mathcal{Q}$ and $j\in\{1,\ldots,N\}$,
\begin{equation}\label{eq:diff_main}
    \operatorname*{ess\,sup}_{\mathbf{y}\in Q^*}
    \int_0^{d(Q)^2}\int_{\mathbb{R}^N}
    |T_{j,\mathbf{x}}(h_t-k_t)(\mathbf{x},\mathbf{y})|
    \,dw(\mathbf{x})\,\frac{dt}{\sqrt{t}}
    \leq C.
\end{equation}
Consequently, for every $f\in L^1(dw)$,
\begin{equation}\label{eq:diff_convergence}
    \lim_{\varepsilon\to0^+}
    \|W_{j,Q}^{\varepsilon}(\phi_Qf)
      -W_{j,Q}(\phi_Qf)\|_{L^1(dw)}
    =0,
\end{equation}
and
\begin{equation}\label{eq:diff_boudedness}
    \|W_{j,Q}(\phi_Qf)\|_{L^1(dw)}
    \leq C\|\phi_Qf\|_{L^1(dw)}.
\end{equation}
\end{lemma}

\begin{proof}
Fix $Q\in\mathcal{Q}$.
By Corollary~\ref{coro:outside}, the contribution
to~\eqref{eq:diff_main} from $\mathbf{x}\notin Q^{**}$
is uniformly bounded. It remains to estimate the integral
over $Q^{**}$.

Set $V_1=V\mathbf{1}_{Q^{***}}$ and $V_2=V-V_1$.
By~\eqref{eq:perturbation_derivative} and Tonelli's theorem,
for almost every $\mathbf{y}\in Q^*$,
\[
\begin{aligned}
    &\int_0^{d(Q)^2}\int_{Q^{**}}
    |T_{j,\mathbf{x}}(h_t-k_t)(\mathbf{x},\mathbf{y})|
    \,dw(\mathbf{x})\,\frac{dt}{\sqrt{t}}
    \leq
    J'_1(\mathbf{y})+J''_1(\mathbf{y})+J_2(\mathbf{y}),
\end{aligned}
\]
where
\[
\begin{aligned}
    J'_1(\mathbf{y})
    &=
    \int_0^{d(Q)^2}\int_0^{t/2}
    \int_{\mathbb{R}^N}
    \left(\int_{Q^{**}}
    |T_{j,\mathbf{x}}h_{t-s}(\mathbf{x},\mathbf{z})|
    \,dw(\mathbf{x})\right)
    V_1(\mathbf{z})k_s(\mathbf{z},\mathbf{y})
    \,dw(\mathbf{z})\,ds\,\frac{dt}{\sqrt{t}},\\
    J''_1(\mathbf{y})
    &=
    \int_0^{d(Q)^2}\int_{t/2}^{t}
    \int_{\mathbb{R}^N}
    \left(\int_{Q^{**}}
    |T_{j,\mathbf{x}}h_{t-s}(\mathbf{x},\mathbf{z})|
    \,dw(\mathbf{x})\right)
    V_1(\mathbf{z})k_s(\mathbf{z},\mathbf{y})
    \,dw(\mathbf{z})\,ds\,\frac{dt}{\sqrt{t}},\\
    J_2(\mathbf{y})
    &=
    \int_0^{d(Q)^2}\int_0^t
    \int_{\mathbb{R}^N}
    \left(\int_{Q^{**}}
    |T_{j,\mathbf{x}}h_{t-s}(\mathbf{x},\mathbf{z})|
    \,dw(\mathbf{x})\right)
    V_2(\mathbf{z})k_s(\mathbf{z},\mathbf{y})
    \,dw(\mathbf{z})\,ds\,\frac{dt}{\sqrt{t}}.
\end{aligned}
\]

For $s<t/2$, \eqref{eq:heat_derivative_moments} gives
\[
    \int_{Q^{**}}
    |T_{j,\mathbf{x}}h_{t-s}(\mathbf{x},\mathbf{z})|
    \,dw(\mathbf{x})\leq Ct^{-1/2}.
\]
Moreover, by Theorem~\ref{theorem:heat}, the domination
$k_s\leq h_s$, and the doubling property,
\[
    k_s(\mathbf{z},\mathbf{y})
    \leq C\mathcal{G}_{s/c}(\mathbf{y},\mathbf{z})
    \leq C'\mathcal{G}_{2s/c}(\mathbf{y},\mathbf{z}).
\]
Thus condition~\eqref{eq:K} yields
\[
\begin{aligned}
    J'_1(\mathbf{y})
    &\leq
    C\int_0^{d(Q)^2}
    \left(
        \int_0^{t/2}\int_{Q^{***}}
        V(\mathbf{z})k_s(\mathbf{z},\mathbf{y})
        \,dw(\mathbf{z})\,ds
    \right)\frac{dt}{t}\leq
    C\int_0^{d(Q)^2}
    \left(\frac{t}{d(Q)^2}\right)^\delta
    \frac{dt}{t}
    \leq C.
\end{aligned}
\]

For $J''_1$, we use~\eqref{eq:heat_derivative_moments},
change the order of integration, and apply
\eqref{eq:square_root_integral} and~\eqref{eq:VK_kernel}:
\[
\begin{aligned}
    J''_1(\mathbf{y})
    &\leq
    C\int_0^{d(Q)^2}\int_{Q^{***}}
    \left(
        \int_s^{\min(2s,d(Q)^2)}
        \frac{dt}{\sqrt{t-s}\sqrt{t}}
    \right)
    V(\mathbf{z})k_s(\mathbf{z},\mathbf{y})
    \,dw(\mathbf{z})\,ds\\
    &\leq
    C\int_0^{d(Q)^2}\int_{Q^{***}}
    V(\mathbf{z})k_s(\mathbf{z},\mathbf{y})
    \,dw(\mathbf{z})\,ds
    \leq C.
\end{aligned}
\]

Finally, if $\mathbf{x}\in Q^{**}$ and
$\mathbf{z}\notin Q^{***}$, then
$\|\mathbf{x}-\mathbf{z}\|\geq d(Q)$. Hence
\eqref{eq:heat_derivative_moments} implies
\[
    \int_{Q^{**}}
    |T_{j,\mathbf{x}}h_u(\mathbf{x},\mathbf{z})|
    \,dw(\mathbf{x})
    \leq
    \frac{1}{d(Q)}\int_{\mathbb{R}^N}
    \|\mathbf{x}-\mathbf{z}\|
    |T_{j,\mathbf{x}}h_u(\mathbf{x},\mathbf{z})|
    \,dw(\mathbf{x})
    \leq \frac{C}{d(Q)}.
\]
Using~\eqref{eq:VK_kernel}, we conclude that
\[
\begin{aligned}
    J_2(\mathbf{y})
    &\leq
    \frac{C}{d(Q)}\int_0^{d(Q)^2}
    \left(
        \int_0^t\int_{\mathbb{R}^N}
        V(\mathbf{z})k_s(\mathbf{z},\mathbf{y})
        \,dw(\mathbf{z})\,ds
    \right)\frac{dt}{\sqrt{t}}\leq
    \frac{C}{d(Q)}\int_0^{d(Q)^2}\frac{dt}{\sqrt{t}}
    \leq C.
\end{aligned}
\]
This proves~\eqref{eq:diff_main}.

Let $g=\phi_Qf$. Since $\operatorname{supp}g\subseteq Q^*$,
estimate~\eqref{eq:diff_main} and Tonelli's theorem show
that
\[
    F(\mathbf{x})
    =
    \int_{\mathbb{R}^N}g(\mathbf{y})
    \int_0^{d(Q)^2}
    T_{j,\mathbf{x}}(h_t-k_t)(\mathbf{x},\mathbf{y})
    \,\frac{dt}{\sqrt{t}}\,dw(\mathbf{y})
\]
is absolutely convergent for almost every $\mathbf{x}$,
belongs to $L^1(dw)$, and satisfies
$\|F\|_{L^1(dw)}\leq C\|g\|_{L^1(dw)}$.
Furthermore, for sufficiently small $\varepsilon>0$,
\[
\begin{aligned}
    \|W_{j,Q}^{\varepsilon}g-F\|_{L^1(dw)}
    &\leq
    \int_{Q^*}|g(\mathbf{y})|
    \int_0^\varepsilon\int_{\mathbb{R}^N}
    |T_{j,\mathbf{x}}(h_t-k_t)(\mathbf{x},\mathbf{y})|
    \,dw(\mathbf{x})\,\frac{dt}{\sqrt{t}}\,dw(\mathbf{y})
    \longrightarrow0
\end{aligned}
\]
by dominated convergence.
Since $W_{j,Q}^{\varepsilon}g$ also converges to
$W_{j,Q}g$ in $\mathcal{S}'$, we have $F=W_{j,Q}g$.
This proves~\eqref{eq:diff_convergence}
and~\eqref{eq:diff_boudedness}.
\end{proof}

\begin{lemma}\label{lem:commutator}
Assume that~\eqref{eq:D} holds.
There is a constant $C>0$ such that for every
$Q\in\mathcal{Q}$, $j\in\{1,\ldots,N\}$, and
$f\in L^1(dw)$ with
\[
    \operatorname{supp}f
    \subseteq
    \bigcup_{Q'\in\mathcal{Q}'(Q)}(Q')^{****},
\]
the distribution
$\widetilde{R}_j(\phi_Qf)-\phi_Q\widetilde{R}_jf$
is represented by a function in $L^1(dw)$ and satisfies
\[
    \|\widetilde{R}_j(\phi_Qf)
      -\phi_Q\widetilde{R}_jf\|_{L^1(dw)}
    \leq C\|f\|_{L^1(dw)}.
\]
\end{lemma}

\begin{proof}
By~\eqref{eq:partition_derivative},
\[
    |\phi_Q(\mathbf{y})-\phi_Q(\mathbf{x})|
    \leq C\frac{\|\mathbf{x}-\mathbf{y}\|}{d(Q)}.
\]
Consequently,~\eqref{eq:weight_schrodinger} with $r=d(Q)$
gives, for almost every $\mathbf{y}\in\mathbb{R}^N$,
\begin{equation}\label{eq:lipschitz}
\begin{aligned}
    &\int_0^{d(Q)^2}\int_{\mathbb{R}^N}
    |T_{j,\mathbf{x}}k_t(\mathbf{x},\mathbf{y})|
    |\phi_Q(\mathbf{y})-\phi_Q(\mathbf{x})|
    \,dw(\mathbf{x})\,\frac{dt}{\sqrt{t}}\\
    &\qquad\leq
    C\int_0^{d(Q)^2}\int_{\mathbb{R}^N}
    \frac{\|\mathbf{x}-\mathbf{y}\|}{d(Q)}
    |T_{j,\mathbf{x}}k_t(\mathbf{x},\mathbf{y})|
    \,dw(\mathbf{x})\,\frac{dt}{\sqrt{t}}
    \leq C.
\end{aligned}
\end{equation}
For almost every
$\mathbf{y}\in\bigcup_{Q'\in\mathcal{Q}'(Q)}(Q')^{****}$,
the proof of Lemma~\ref{lem:R_j_infty} also yields
\[
    \int_{d(Q)^2}^{\infty}\int_{\mathbb{R}^N}
    |T_{j,\mathbf{x}}k_t(\mathbf{x},\mathbf{y})|
    \,dw(\mathbf{x})\,\frac{dt}{\sqrt{t}}
    \leq C.
\]
Since $0\leq\phi_Q\leq1$, combining these estimates gives
\[
    \int_0^\infty\int_{\mathbb{R}^N}
    |T_{j,\mathbf{x}}k_t(\mathbf{x},\mathbf{y})|
    |\phi_Q(\mathbf{y})-\phi_Q(\mathbf{x})|
    \,dw(\mathbf{x})\,\frac{dt}{\sqrt{t}}
    \leq C
\]
for almost every $\mathbf{y}$ in this union.

Thus, by Tonelli's theorem,
\[
    F(\mathbf{x})
    =
    \int_{\mathbb{R}^N}\int_0^\infty
    T_{j,\mathbf{x}}k_t(\mathbf{x},\mathbf{y})
    [\phi_Q(\mathbf{y})-\phi_Q(\mathbf{x})]
    f(\mathbf{y})\,\frac{dt}{\sqrt{t}}\,dw(\mathbf{y})
\]
is absolutely convergent for almost every $\mathbf{x}$,
belongs to $L^1(dw)$, and satisfies
\[
    \|F\|_{L^1(dw)}\leq C\|f\|_{L^1(dw)}.
\]
Moreover, the integral representation of the truncated
operators and dominated convergence imply
\[
    \widetilde{R}_j^\varepsilon(\phi_Qf)
    -\phi_Q\widetilde{R}_j^\varepsilon f
    \longrightarrow F
    \qquad\text{in }L^1(dw).
\]
On the other hand, the left-hand side converges in
$\mathcal{S}'$ to
$\widetilde{R}_j(\phi_Qf)-\phi_Q\widetilde{R}_jf$,
since multiplication by $\phi_Q\in C_c^\infty(\mathbb{R}^N)$
is continuous on $\mathcal{S}'$.
Therefore,
\[
    F=\widetilde{R}_j(\phi_Qf)-\phi_Q\widetilde{R}_jf
\]
in the sense of distributions, which proves the claim.
\end{proof}

\begin{lemma}\label{lem:sum_outside}
Assume that~\eqref{eq:D} holds.
There is a constant $C>0$ such that for every
$f\in L^1(dw)$ and $j\in\{1,\ldots,N\}$,
\begin{equation}\label{eq:sum_outside}
    \sum_{Q\in\mathcal{Q}}
    \left\|
        \chi_{Q^{***}}\widetilde{R}_j
        \Big(
            \sum_{Q''\in\mathcal{Q}''(Q)}\phi_{Q''}f
        \Big)
    \right\|_{L^1(dw)}
    \leq C\|f\|_{L^1(dw)}.
\end{equation}
Here each restricted distribution is represented by an
integrable function on the interior of $Q^{***}$,
and the product with $\chi_{Q^{***}}$ denotes its
extension by zero.
\end{lemma}

\begin{proof}
By Corollary~\ref{coro:outside} and the large-time
estimate established in the proof of
Lemma~\ref{lem:R_j_infty}, for every $P\in\mathcal{Q}$
and almost every $\mathbf{y}\in P^*$,
\[
    \int_0^\infty\int_{\mathbb{R}^N\setminus P^{**}}
    |T_{j,\mathbf{x}}k_t(\mathbf{x},\mathbf{y})|
    \,dw(\mathbf{x})\,\frac{dt}{\sqrt{t}}
    \leq C.
\]
Consequently, the nonnegative function
\[
    A_P(\mathbf{x})
    =
    \int_{P^*}|\phi_P(\mathbf{y})f(\mathbf{y})|
    \int_0^\infty
    |T_{j,\mathbf{x}}k_t(\mathbf{x},\mathbf{y})|
    \,\frac{dt}{\sqrt{t}}\,dw(\mathbf{y})
\]
satisfies
\[
    \int_{\mathbb{R}^N\setminus P^{**}}
    A_P(\mathbf{x})\,dw(\mathbf{x})
    \leq C\|\phi_Pf\|_{L^1(dw)}.
\]

By the definition of $\mathcal{Q}''(Q)$,
\[
    P\in\mathcal{Q}''(Q)
    \quad\Longleftrightarrow\quad
    Q^{***}\cap P^{***}=\emptyset
    \quad\Longleftrightarrow\quad
    Q\in\mathcal{Q}''(P).
\]
Thus, when interchanging the sums over $Q$ and $P$,
the condition $P\in\mathcal{Q}''(Q)$ can be replaced
by $Q\in\mathcal{Q}''(P)$.
Furthermore, since $P^{**}\subseteq P^{***}$,
each $Q\in\mathcal{Q}''(P)$ satisfies
\[
    Q^{***}\subseteq\mathbb{R}^N\setminus P^{**}.
\]
Combining this inclusion with the bounded overlap of
$\{Q^{***}:Q\in\mathcal{Q}\}$, which follows
from~\eqref{eq:finite_overlap}, we obtain
\[
    \sum_{Q\in\mathcal{Q}''(P)}
    \chi_{Q^{***}}(\mathbf{x})
    \leq
    C\chi_{\mathbb{R}^N\setminus P^{**}}(\mathbf{x}).
\]
Therefore, Tonelli's theorem yields
\[
\begin{aligned}
    &\sum_{Q\in\mathcal{Q}}
    \int_{Q^{***}}
    \sum_{P\in\mathcal{Q}''(Q)}
    A_P(\mathbf{x})\,dw(\mathbf{x})=
    \sum_{P\in\mathcal{Q}}
    \int_{\mathbb{R}^N}
    A_P(\mathbf{x})
    \sum_{Q\in\mathcal{Q}''(P)}
    \chi_{Q^{***}}(\mathbf{x})\,dw(\mathbf{x})\\
    &\qquad\leq
    C\sum_{P\in\mathcal{Q}}
    \int_{\mathbb{R}^N\setminus P^{**}}
    A_P(\mathbf{x})\,dw(\mathbf{x})\leq
    C\sum_{P\in\mathcal{Q}}\|\phi_Pf\|_{L^1(dw)}
    =C\|f\|_{L^1(dw)}.
\end{aligned}
\]

For each $Q\in\mathcal{Q}$, set
\[
    g_Q=\sum_{P\in\mathcal{Q}''(Q)}\phi_Pf,
\]
where the series converges in $L^1(dw)$.
The preceding estimate shows that
\[
    F_Q(\mathbf{x})
    =
    \int_{\mathbb{R}^N}\int_0^\infty
    T_{j,\mathbf{x}}k_t(\mathbf{x},\mathbf{y})
    g_Q(\mathbf{y})\,\frac{dt}{\sqrt{t}}\,dw(\mathbf{y})
\]
is absolutely convergent for almost every
$\mathbf{x}\in Q^{***}$ and satisfies
\[
    \sum_{Q\in\mathcal{Q}}
    \|F_Q\|_{L^1(Q^{***},dw)}
    \leq C\|f\|_{L^1(dw)}.
\]
By dominated convergence,
\[
    \widetilde{R}_j^\varepsilon g_Q
    \longrightarrow F_Q
    \qquad\text{in }L^1(Q^{***},dw).
\]
Since
$\widetilde{R}_j^\varepsilon g_Q
\to\widetilde{R}_jg_Q$ in $\mathcal{S}'$,
testing against functions in
$C_c^\infty(\operatorname{int}Q^{***})$
identifies $F_Q$ with the restriction of
$\widetilde{R}_jg_Q$ to this open set.
The desired estimate follows.
\end{proof}

\section{Proof of Theorem\texorpdfstring{~\ref{theo:main}}{1}}

\subsection{Proof of the first inequality}

Let $f\in L^1(dw)$ and assume that
$\widetilde{R}_jf\in L^1(dw)$ for $j=1,\ldots,N$.
All operator identities below are initially understood
in $\mathcal{S}'$.

By~\eqref{eq:W_other_formula},
Lemma~\ref{lem:diff}, and the identity
$\sum_{Q\in\mathcal{Q}}\|\phi_Qf\|_{L^1(dw)}
=\|f\|_{L^1(dw)}$, we have
\begin{equation}\label{eq:pattern_1}
\begin{aligned}
    \sum_{Q\in\mathcal{Q}}
    \|R_j^{d(Q)}(\phi_Qf)\|_{L^1(dw)}
    &\leq
    \sum_{Q\in\mathcal{Q}}
    \|W_{j,Q}(\phi_Qf)\|_{L^1(dw)}
    +
    \sum_{Q\in\mathcal{Q}}
    \|\widetilde{R}_{j,Q,0}(\phi_Qf)\|_{L^1(dw)}\\
    &\leq
    C\|f\|_{L^1(dw)}
    +
    \sum_{Q\in\mathcal{Q}}
    \|\widetilde{R}_{j,Q,0}(\phi_Qf)\|_{L^1(dw)}.
\end{aligned}
\end{equation}

For each $Q\in\mathcal{Q}$, set
\[
    f_Q^{\mathrm{near}}
    =\sum_{P\in\mathcal{Q}'(Q)}\phi_Pf,
    \qquad
    f_Q^{\mathrm{far}}
    =\sum_{P\in\mathcal{Q}''(Q)}\phi_Pf.
\]
These sums converge in $L^1(dw)$ and satisfy
$f=f_Q^{\mathrm{near}}+f_Q^{\mathrm{far}}$.
Since $\phi_Q\phi_P=0$ for
$P\in\mathcal{Q}''(Q)$, we also have
$\phi_Qf=\phi_Qf_Q^{\mathrm{near}}$.
Consequently, in $\mathcal{S}'$ we have
\[
\begin{aligned}
    \widetilde{R}_{j,Q,0}(\phi_Qf)
    &=
    -\widetilde{R}_{j,Q,\infty}(\phi_Qf)\\
    &\quad+
    \bigl[
        \widetilde{R}_j(\phi_Qf_Q^{\mathrm{near}})
        -\phi_Q\widetilde{R}_jf_Q^{\mathrm{near}}
    \bigr]\\
    &\quad-\phi_Q\widetilde{R}_jf_Q^{\mathrm{far}}
    +\phi_Q\widetilde{R}_jf.
\end{aligned}
\]
We estimate these four terms separately.

For the large-time term, Lemma~\ref{lem:R_j_infty}
and $\operatorname{supp}(\phi_Qf)\subseteq Q^*$ give
\[
    \sum_{Q\in\mathcal{Q}}
    \|\widetilde{R}_{j,Q,\infty}(\phi_Qf)\|_{L^1(dw)}
    \leq C\sum_{Q\in\mathcal{Q}}
    \|\phi_Qf\|_{L^1(dw)}
    =C\|f\|_{L^1(dw)}.
\]

For the commutator term, Lemma~\ref{lem:commutator}
yields
\[
\begin{aligned}
    &\sum_{Q\in\mathcal{Q}}
    \bigl\|
        \widetilde{R}_j(\phi_Qf_Q^{\mathrm{near}})
        -\phi_Q\widetilde{R}_jf_Q^{\mathrm{near}}
    \bigr\|_{L^1(dw)}\leq
    C\sum_{Q\in\mathcal{Q}}
    \|f_Q^{\mathrm{near}}\|_{L^1(dw)}
    \leq C\|f\|_{L^1(dw)}.
\end{aligned}
\]
Indeed, condition~\eqref{eq:finite_overlap} implies
$\sup_{P\in\mathcal{Q}}\#\mathcal{Q}'(P)<\infty$.
Using the symmetry
$P\in\mathcal{Q}'(Q)\Longleftrightarrow
Q\in\mathcal{Q}'(P)$ and the nonnegativity of $\phi_P$,
we obtain
\[
    \sum_{Q\in\mathcal{Q}}
    \|f_Q^{\mathrm{near}}\|_{L^1(dw)}
    =
    \sum_{P\in\mathcal{Q}}
    \#\mathcal{Q}'(P)\,\|\phi_Pf\|_{L^1(dw)}
    \leq C\|f\|_{L^1(dw)}.
\]

For the far term, Lemma~\ref{lem:sum_outside}
and $0\leq\phi_Q\leq\chi_{Q^{***}}$ imply
\[
    \sum_{Q\in\mathcal{Q}}
    \|\phi_Q\widetilde{R}_jf_Q^{\mathrm{far}}\|_{L^1(dw)}
    \leq
    \sum_{Q\in\mathcal{Q}}
    \|\chi_{Q^{***}}\widetilde{R}_jf_Q^{\mathrm{far}}
    \|_{L^1(dw)}
    \leq C\|f\|_{L^1(dw)}.
\]

Finally, since $\widetilde{R}_jf\in L^1(dw)$
and $\{\phi_Q\}$ is a nonnegative partition of unity,
\[
    \sum_{Q\in\mathcal{Q}}
    \|\phi_Q\widetilde{R}_jf\|_{L^1(dw)}
    =\|\widetilde{R}_jf\|_{L^1(dw)}.
\]
Thus each of the four terms is represented by an
$L^1(dw)$ function. Combining the preceding estimates
with the triangle inequality yields
\begin{equation}\label{eq:pattern_2}
\begin{aligned}
    \sum_{Q\in\mathcal{Q}}
    \|\widetilde{R}_{j,Q,0}(\phi_Qf)\|_{L^1(dw)}
    &\leq
    C\|f\|_{L^1(dw)}
    +\|\widetilde{R}_jf\|_{L^1(dw)}.
\end{aligned}
\end{equation}

Combining~\eqref{eq:pattern_1},
\eqref{eq:pattern_2}, and
Proposition~\ref{propo:local_Riesz}, we obtain
\begin{equation}\label{eq:sum_local_Hardy}
    \sum_{Q\in\mathcal{Q}}
    \|\phi_Qf\|_{H^1_{\mathrm{loc},d(Q)}}
    \leq
    C\left(
        \|f\|_{L^1(dw)}
        +\sum_{j=1}^N
        \|\widetilde{R}_jf\|_{L^1(dw)}
    \right).
\end{equation}

We next verify the localization needed to pass from
local atoms to $H_{\mathcal{Q}}^{1,\mathrm{at}}$ atoms.
For each $Q$, choose
$\psi_Q\in C_c^\infty(\mathbb{R}^N)$ such that
\[
    0\leq\psi_Q\leq1,
    \qquad
    \psi_Q=1\ \text{on }Q^*,
    \qquad
    \operatorname{supp}\psi_Q\subseteq Q^{**},
    \qquad
    \|\nabla\psi_Q\|_\infty\leq C/d(Q).
\]
We claim that for every
$H^1_{\mathrm{loc},d(Q)}$ atom $a$, the function
$\psi_Qa$ admits a decomposition into
$H_{\mathcal{Q}}^{1,\mathrm{at}}$ atoms associated with
$Q$, with the sum of the absolute values of the
coefficients bounded uniformly.

Suppose that $a$ is supported in $B(\mathbf{x}_0,r)$.
First assume that $r\geq d(Q)/4$.
Cover $Q^{**}$ by a uniformly bounded number of balls
$B_i=B(\mathbf{x}_i,d(Q))$, with
$\mathbf{x}_i\in Q^{**}$, and choose a measurable
partition $\{E_i\}$ of $Q^{**}$ subordinate to this cover.
Each $B_i$ is contained in $Q^{****}$.
If $\mathbf{1}_{E_i}\psi_Qa$ is nonzero, then
$B_i\cap B(\mathbf{x}_0,r)\neq\emptyset$.
Since $d(Q)\leq4r$, the doubling property gives
\[
    w(B_i)\leq Cw(B(\mathbf{x}_0,r)).
\]
It follows that each nonzero
$\mathbf{1}_{E_i}\psi_Qa$ is a multiple of an atom
associated with $Q$, with coefficient at most $C$.
No cancellation is required, because the supporting
balls have radius $d(Q)$.

Now assume that $r<d(Q)/4$ and $\psi_Qa\neq0$.
Set
\[
    m=\int_{\mathbb{R}^N}\psi_Qa\,dw,
    \qquad
    b_\rho
    =\frac{\chi_{B(\mathbf{x}_0,\rho)}}
           {w(B(\mathbf{x}_0,\rho))}.
\]
Since $\int a\,dw=0$,
\[
    |m|
    =
    \left|
        \int_{\mathbb{R}^N}
        [\psi_Q(\mathbf{x})-\psi_Q(\mathbf{x}_0)]
        a(\mathbf{x})\,dw(\mathbf{x})
    \right|
    \leq C\frac{r}{d(Q)}.
\]
Also $|m|\leq1$, so
$\psi_Qa-mb_r$ has integral zero and
$L^\infty$ norm at most
$2w(B(\mathbf{x}_0,r))^{-1}$.

Let $M$ be the smallest integer such that
$2^Mr\geq d(Q)$, and set
\[
    \rho_k=2^kr \quad (0\leq k<M),
    \qquad
    \rho_M=d(Q).
\]
Then
\[
    \psi_Qa
    =
    (\psi_Qa-mb_r)
    +m\sum_{k=0}^{M-1}(b_{\rho_k}-b_{\rho_{k+1}})
    +mb_{d(Q)}.
\]
The first term is a multiple of an atom with
coefficient at most $2$.
Each difference in the sum has integral zero and,
by the doubling property, is a multiple of an atom
supported in $B(\mathbf{x}_0,\rho_{k+1})$,
with a uniformly bounded coefficient.
The final term is a multiple of a noncancellative
atom supported in $B(\mathbf{x}_0,d(Q))$.

All these supporting balls are contained in $Q^{****}$.
Indeed, $B(\mathbf{x}_0,r)$ meets $Q^{**}$,
$r<d(Q)/4$, and their radii do not exceed $d(Q)$.
The sum of the absolute values of the coefficients
is bounded by
\[
    C+C|m|(M+1)
    \leq
    C+C\frac{r}{d(Q)}
    \left(1+\log\frac{d(Q)}r\right)
    \leq C.
\]
This proves the claim.

For each $Q\in\mathcal{Q}$ with $\phi_Qf\neq0$, choose
an atomic decomposition
\[
    \phi_Qf=\sum_{\nu=1}^{\infty}\lambda_{\nu,Q}b_{\nu,Q}
    \qquad\text{in }L^1(dw),
\]
where $b_{\nu,Q}$ are $H^1_{\mathrm{loc},d(Q)}$ atoms and
\[
    \sum_{\nu=1}^{\infty}|\lambda_{\nu,Q}|
    \leq 2\|\phi_Qf\|_{H^1_{\mathrm{loc},d(Q)}}.
\]
If $\phi_Qf=0$, we use the empty decomposition.
Since multiplication by $\psi_Q$ is bounded on $L^1(dw)$
and $\psi_Q\phi_Qf=\phi_Qf$, we have
\[
    \phi_Qf
    =\sum_{\nu=1}^{\infty}
    \lambda_{\nu,Q}\psi_Qb_{\nu,Q}
    \qquad\text{in }L^1(dw).
\]

By the localization argument above, for each $\nu$
there is a finite decomposition
\[
    \psi_Qb_{\nu,Q}
    =
    \sum_{k=1}^{M_{\nu,Q}}
    \mu_{\nu,k,Q}a_{\nu,k,Q},
    \qquad
    \sum_{k=1}^{M_{\nu,Q}}|\mu_{\nu,k,Q}|\leq C,
\]
where each $a_{\nu,k,Q}$ is an
$H_{\mathcal{Q}}^{1,\mathrm{at}}$ atom associated with $Q$.
The constant $C$ is independent of $\nu$ and $Q$.
Consequently,
\[
    \phi_Qf
    =
    \sum_{\nu=1}^{\infty}
    \sum_{k=1}^{M_{\nu,Q}}
    c_{\nu,k,Q}a_{\nu,k,Q},
    \qquad
    c_{\nu,k,Q}
    =\lambda_{\nu,Q}\mu_{\nu,k,Q}.
\]
These coefficients satisfy
\[
\begin{aligned}
    \sum_{\nu=1}^{\infty}
    \sum_{k=1}^{M_{\nu,Q}}|c_{\nu,k,Q}|
    &=
    \sum_{\nu=1}^{\infty}|\lambda_{\nu,Q}|
    \sum_{k=1}^{M_{\nu,Q}}|\mu_{\nu,k,Q}|\leq
    C\sum_{\nu=1}^{\infty}|\lambda_{\nu,Q}|\leq
    2C\|\phi_Qf\|_{H^1_{\mathrm{loc},d(Q)}}.
\end{aligned}
\]
Since every atom has $L^1(dw)$ norm at most one,
this double series converges absolutely in $L^1(dw)$.

Summing over $Q$ and using~\eqref{eq:sum_local_Hardy},
we obtain
\[
\begin{aligned}
    \sum_{Q\in\mathcal{Q}}
    \sum_{\nu=1}^{\infty}
    \sum_{k=1}^{M_{\nu,Q}}|c_{\nu,k,Q}|
    &\leq
    C\sum_{Q\in\mathcal{Q}}
    \|\phi_Qf\|_{H^1_{\mathrm{loc},d(Q)}}\leq
    C\left(
        \|f\|_{L^1(dw)}
        +\sum_{j=1}^N
        \|\widetilde{R}_jf\|_{L^1(dw)}
    \right).
\end{aligned}
\]
Thus the resulting series is absolutely convergent
in $L^1(dw)$ and, since $f=\sum_Q\phi_Qf$ in $L^1(dw)$,
represents $f$:
\[
    f=
    \sum_{Q\in\mathcal{Q}}
    \sum_{\nu=1}^{\infty}
    \sum_{k=1}^{M_{\nu,Q}}
    c_{\nu,k,Q}a_{\nu,k,Q}.
\]
By the definition of the atomic norm,
\[
    \|f\|_{H_{\mathcal{Q}}^{1,\mathrm{at}}}
    \leq
    C\left(
        \|f\|_{L^1(dw)}
        +\sum_{j=1}^N
        \|\widetilde{R}_jf\|_{L^1(dw)}
    \right).
\]
Theorem~\ref{theorem:main_prev} now gives the first
inequality.

\subsection{Proof of the second inequality}

We first extend the estimate in Lemma~\ref{lem:diff}
to functions supported in $Q^{****}$.
We claim that
\begin{equation}\label{eq:diff_enlarged}
    \operatorname*{ess\,sup}_{\mathbf{y}\in Q^{****}}
    \int_0^{d(Q)^2}\int_{\mathbb{R}^N}
    |T_{j,\mathbf{x}}(h_t-k_t)(\mathbf{x},\mathbf{y})|
    \,dw(\mathbf{x})\,\frac{dt}{\sqrt{t}}
    \leq C.
\end{equation}

To prove this, choose $P\in\mathcal{Q}$ containing
$\mathbf{y}\in Q^{****}$.
Then $\mathbf{y}\in P^*$ and
$P^{****}\cap Q^{****}\neq\emptyset$.
Hence~\eqref{eq:finite_overlap} gives
$d(P)\sim d(Q)$.
For almost every such $\mathbf{y}$,
\[
\begin{aligned}
    &\int_0^{d(Q)^2}\int_{\mathbb{R}^N}
    |T_{j,\mathbf{x}}(h_t-k_t)(\mathbf{x},\mathbf{y})|
    \,dw(\mathbf{x})\,\frac{dt}{\sqrt{t}}\\
    &\quad\leq
    \int_0^{d(P)^2}\int_{\mathbb{R}^N}
    |T_{j,\mathbf{x}}(h_t-k_t)(\mathbf{x},\mathbf{y})|
    \,dw(\mathbf{x})\,\frac{dt}{\sqrt{t}}\\
    &\qquad+
    \int_{\min(d(P)^2,d(Q)^2)}^{\max(d(P)^2,d(Q)^2)}
    \int_{\mathbb{R}^N}
    \left(
        |T_{j,\mathbf{x}}h_t(\mathbf{x},\mathbf{y})|
        +|T_{j,\mathbf{x}}k_t(\mathbf{x},\mathbf{y})|
    \right)
    \,dw(\mathbf{x})\,\frac{dt}{\sqrt{t}}.
\end{aligned}
\]
The first term is bounded by Lemma~\ref{lem:diff}.
The second is uniformly bounded by
\eqref{eq:heat_derivative_moments} and
Corollary~\ref{coro:logarithm}, since its interval of
integration can be covered by a uniformly bounded
number of intervals of the form $[u,2u]$.
As $\mathcal{Q}$ is countable, this proves
\eqref{eq:diff_enlarged}.

By Tonelli's theorem and dominated convergence,
\eqref{eq:diff_enlarged} implies that for every
$g\in L^1(dw)$ supported in $Q^{****}$,
the truncated difference operators converge in
$L^1(dw)$ and
\[
    \|W_{j,Q}g\|_{L^1(dw)}
    \leq C\|g\|_{L^1(dw)}.
\]
The limit is identified with $W_{j,Q}g$ by its
distributional definition, exactly as in the proof
of Lemma~\ref{lem:diff}.

Let $a$ be an $H_{\mathcal{Q}}^{1,\mathrm{at}}$ atom
associated with $Q$.
By Definition~\ref{def:atomic},
$\operatorname{supp}a\subseteq Q^{****}$ and
$\|a\|_{L^1(dw)}\leq1$.
Moreover, $a$ is also an
$H^1_{\mathrm{loc},d(Q)}$ atom, so
Proposition~\ref{propo:local_Riesz} gives
\[
    \|R_j^{d(Q)}a\|_{L^1(dw)}\leq C.
\]
Since $Q\in\mathcal{Q}'(Q)$,
Lemma~\ref{lem:R_j_infty} also yields
\[
    \|\widetilde{R}_{j,Q,\infty}a\|_{L^1(dw)}
    \leq C.
\]
Using the distributional identity
\[
    \widetilde{R}_ja
    =
    \widetilde{R}_{j,Q,\infty}a
    +R_j^{d(Q)}a-W_{j,Q}a,
\]
we conclude that $\widetilde{R}_ja$ is represented
by a function in $L^1(dw)$ and
\begin{equation}\label{eq:uniform_on_atoms}
    \|\widetilde{R}_ja\|_{L^1(dw)}
    \leq C.
\end{equation}

Finally, let $f\in H_L^1$.
By Theorem~\ref{theorem:main_prev}, there is an
atomic decomposition
\[
    f=\sum_{m=1}^{\infty}c_ma_m
    \qquad\text{in }L^1(dw),
    \qquad
    \sum_{m=1}^{\infty}|c_m|
    \leq C\|f\|_{H_L^1}.
\]
For $M\geq1$, set
\[
    f_M=\sum_{m=1}^{M}c_ma_m,
    \qquad
    G_{j,M}=\sum_{m=1}^{M}c_m\widetilde{R}_ja_m.
\]
By linearity for finite sums,
$G_{j,M}=\widetilde{R}_jf_M$ in $\mathcal{S}'$.
The uniform estimate on atoms~\eqref{eq:uniform_on_atoms} gives, for $M'>M$,
\[
    \|G_{j,M'}-G_{j,M}\|_{L^1(dw)}
    \leq C\sum_{m=M+1}^{M'}|c_m|.
\]
Thus $\{G_{j,M}\}_{M\geq1}$ is a Cauchy sequence
in $L^1(dw)$. Denote its limit by $G_j$.
Then
\[
    \|G_j\|_{L^1(dw)}
    \leq C\sum_{m=1}^{\infty}|c_m|.
\]

We now identify $G_j$ with the distribution
$\widetilde{R}_jf$.
This step uses only~\eqref{eq:distribution},
which was established for arbitrary $L^1(dw)$
functions independently of the present theorem.
For every $\varphi\in\mathcal{S}(\mathbb{R}^N)$,
\[
\begin{aligned}
    \left|
        \langle\widetilde{R}_jf,\varphi\rangle
        -\int_{\mathbb{R}^N}G_j(\mathbf{x})
            \varphi(\mathbf{x})\,dw(\mathbf{x})
    \right|
    &\leq
    \left|
        \langle\widetilde{R}_j(f-f_M),\varphi\rangle
    \right|\\
    &\quad+
    \left|
        \int_{\mathbb{R}^N}
        (G_{j,M}-G_j)(\mathbf{x})
        \varphi(\mathbf{x})\,dw(\mathbf{x})
    \right|\\
    &\leq
    C\|f-f_M\|_{L^1(dw)}
    \left(
        \|\varphi\|_{L^2(dw)}
        +\|T_j\varphi\|_{L^\infty}
    \right)\\
    &\quad+
    \|G_{j,M}-G_j\|_{L^1(dw)}
    \|\varphi\|_{L^\infty}.
\end{aligned}
\]
Both terms tend to zero as $M\to\infty$.
Indeed, $G_{j,M}\to G_j$ in $L^1(dw)$, and
\[
    \|f-f_M\|_{L^1(dw)}
    \leq\sum_{m>M}|c_m|\longrightarrow0.
\]
Therefore $\widetilde{R}_jf$ is represented by $G_j$,
and hence
\[
    \|\widetilde{R}_jf\|_{L^1(dw)}
    =\|G_j\|_{L^1(dw)}
    \leq C\sum_{m=1}^{\infty}|c_m|
    \leq C\|f\|_{H_L^1}.
\]
Finally, since $\|a_m\|_{L^1(dw)}\leq1$,
\[
    \|f\|_{L^1(dw)}
    \leq\sum_{m=1}^{\infty}|c_m|
    \leq C\|f\|_{H_L^1}.
\]
Summing over $j=1,\ldots,N$, we obtain
\[
    \|f\|_{L^1(dw)}
    +\sum_{j=1}^N\|\widetilde{R}_jf\|_{L^1(dw)}
    \leq C\|f\|_{H_L^1}.
\]
This completes the proof of Theorem~\ref{theo:main}.

\end{document}